\documentclass{siamart171218}
\usepackage{amsmath}
\usepackage{amssymb}
\usepackage{enumerate}
\usepackage{enumitem}
\usepackage{matlab-prettifier}
\usepackage{algorithm,algorithmic}
\usepackage{caption}

\usepackage{pgfplots, pgfplotstable, booktabs, colortbl, array}
\usepackage{multirow}

\usepackage{float}
\floatstyle{plaintop}
\restylefloat{table}

\newcommand{\IFTHENELSE}[3]{%
    \STATE\algorithmicif\ {#1} \algorithmicthen\ {#2} \algorithmicelse\ {#3} \algorithmicend\ \algorithmicif%
}

\usepackage{todonotes}

\DeclareMathOperator*{\argmin}{arg\,min}
\newcommand{\ignore}[1]{}

\makeatletter
\newenvironment{breakablealgorithm}
  {
   \begin{center}
     \refstepcounter{algorithm}
     \hrule height.8pt depth0pt \kern2pt
     \renewcommand{\caption}[2][\relax]{
       {\raggedright\textbf{\ALG@name~\thealgorithm} ##2\par}%
       \ifx\relax##1\relax 
         \addcontentsline{loa}{algorithm}{\protect\numberline{\thealgorithm}##2}%
       \else 
         \addcontentsline{loa}{algorithm}{\protect\numberline{\thealgorithm}##1}%
       \fi
       \kern2pt\hrule\kern2pt
     }
  }{
     \kern2pt\hrule\relax
   \end{center}
  }
\makeatother

\begin{document}

\title{A Backward Stable Algorithm for Computing the CS Decomposition via the Polar Decomposition}

\author{Evan S. Gawlik\thanks{Department of Mathematics, University of California, San Diego (\email{egawlik@ucsd.edu})} \and Yuji Nakatsukasa\thanks{National Institute of Informatics, 2-1-2 Hitotsubashi, Chiyoda-ku, Tokyo 101-8430, Japan (\email{nakatsukasa@nii.ac.jp})} \and Brian D. Sutton\thanks{Department of Mathematics, Randolph-Macon College, Ashland, VA 23005 (\email{bsutton@rmc.edu})}}

\date{}

\headers{Backward Stable Algorithm for the CS Decomposition}{E. S. Gawlik, Y. Nakatsukasa, and B. D. Sutton}

\maketitle

\begin{abstract}
We introduce a backward stable algorithm for computing the CS decomposition of a partitioned $2n \times n$ matrix with orthonormal columns, or a rank-deficient partial isometry. 
 The algorithm computes two $n \times n$ polar decompositions (which can be carried out in parallel) followed by an eigendecomposition of a judiciously crafted $n \times n$ Hermitian matrix.  We prove that the algorithm is backward stable whenever the aforementioned decompositions are computed in a backward stable way. 
 Since the polar decomposition and the symmetric eigendecomposition are highly amenable to parallelization, the algorithm inherits this feature.  We illustrate this fact by invoking recently developed algorithms for the polar decomposition and symmetric eigendecomposition that leverage Zolotarev's best rational approximations of the sign function. Numerical examples demonstrate that the resulting algorithm for computing the CS decomposition enjoys excellent numerical stability.
\end{abstract}


\section{Introduction} \label{sec:intro}

The CS decomposition~\cite[Section 2.5.4]{golub2012matrix} allows any partitioned $2n \times n$ matrix
\[
A = \begin{pmatrix} A_1 \\ A_2 \end{pmatrix}, \quad 
A_1, \, A_2 \in \mathbb{C}^{n\times n}
\]
with orthonormal columns to be factorized as
\[
A = \begin{pmatrix} U_1 & 0 \\ 0 & U_2 \end{pmatrix} \begin{pmatrix} C \\ S \end{pmatrix} V_1^*,
\]
where 
 $U_1,U_2,V_1 \in \mathbb{C}^{n \times n}$ are unitary matrices
 and $C,S \in \mathbb{C}^{n \times n}$ are diagonal matrices with nonnegative entries satisfying $C^2+S^2=I$.  In other words, $A_1 = U_1 C V_1^*$ and $A_2 = U_2 S V_1^*$ have highly correlated singular value decompositions: they share the same right singular vectors, and the singular values of $A_1$ and $A_2$ are the cosines and sines, respectively, of angles $0 \le \theta_1 \le \theta_2 \le \dots \le \theta_n \le \tfrac{\pi}{2}$.  An analogous factorization holds for $(m_1+m_2)\times n$ $(m_1,m_2\geq n)$ matrices with orthonormal columns~\cite[Section 2.5.4]{golub2012matrix}.

By writing
\begin{equation} \label{A1decomp}
A_1 = (U_1 V_1^*) (V_1 C V_1^*)
\end{equation}
and
\begin{equation} \label{A2decomp}
A_2 = (U_2 V_1^*) (V_1 S V_1^*),
\end{equation}
another perspective emerges.  Since $W_1 := U_1 V_1^*$ and $W_2 := U_2 V_1^*$ are unitary and $H_1 := V_1 C V_1^*$ and $H_2 := V_1 S V_1^*$ are Hermitian positive semidefinite, the \emph{polar decompositions} $A_i = W_i H_i$, $i=1,2$, are highly correlated.  Specifically, the matrices $H_1$ and $H_2$ are simultaneously diagonalizable with eigenvalues $\{\cos \theta_i\}_{i=1}^n$ and $\{\sin \theta_i\}_{i=1}^n$, respectively.  

In this paper, we leverage the preceding observation to construct a backward stable algorithm for the CS decomposition.  The algorithm computes two polar decompositions $A_1 = W_1 H_1$ and $A_2 = W_2 H_2$ followed by an eigendecomposition of a judiciously crafted Hermitian matrix $B \in \mathbb{C}^{n \times n}$.  As it turns out, the stability of the algorithm depends critically on the choice of $B$. Obvious candidates, such as $H_1$, $H_2$, or even $H_1+H_2$, lead to unstable algorithms.  The choice $B=H_2-H_1$, on the other hand, leads to a backward stable algorithm, assuming that the two polar decompositions and the eigendecomposition are computed in a backward stable way.  A central aim of this paper is to prove this assertion.

One of the hallmarks of this approach is its simplicity: it is built entirely from a pair of standard matrix decompositions for which a wealth of highly optimized algorithms are available.  
In particular, one can compute the CS decomposition in a parallel, communication-efficient way by invoking off-the-shelf algorithms for the polar decomposition and symmetric eigendecomposition.
We illustrate this fact by invoking recently developed algorithms that leverage Zolotarev's best rational approximations of the sign function~\cite{nakatsukasa2016computing}.

With a small modification, our algorithm enables the computation of a generalization of the CS decomposition that is applicable to \emph{partial isometries}.  Recall that $A \in \mathbb{C}^{m \times n}$ is a partial isometry if $AA^*A=A$; equivalently, every singular value of $A$ is either 1 or 0.  We introduce this generalized CS decomposition in Section~\ref{sec:prelim} and prove backward stability of the algorithm in this generalized setting.

Stably computing the CS decomposition of a matrix with orthonormal columns is a notoriously delicate task.  The difficulties stem from the fact that the columns of $V_1$ serve simultaneously as the right singular vectors of $A_1$ \emph{and} $A_2$.  In the presence of roundoff errors, choosing the columns of $V_1$ to satisfy both roles simultaneously is nontrivial, particularly when $\{\theta_i\}_{i=1}^n$ contains clusters of nearby angles.  
Early algorithms for the CS decomposition include~\cite{stewart1982computing,van1985computing}; both algorithms obtain $V_1$ by computing an SVD of either $A_1$ or $A_2$ and then modifying it.  Recent algorithms have focused on simultaneously diagonalizing $A_1$ and $A_2$, using either simultaneous QR iteration or a divide-and-conquer strategy after $A_1$ and $A_2$ have been simultaneously reduced to bidiagonal form~\cite{sutton2009computing,sutton2012stable,sutton2013divide}.  There are also general-purpose algorithms for computing the generalized singular value decomposition~\cite{bai1993computing}, of which the CS decomposition is a special case.

Applications of the CS decomposition are widespread.  It can be used to help compute principal angles between subspaces~\cite[Section 12.4]{golub2012matrix}, the logarithm on the Grassmannian manifold~\cite{gallivan2003efficient}, the generalized singular value decomposition~\cite[Section 8.7.3]{golub2012matrix}, and decompositions of quantum circuits~\cite{tucci1999rudimentary}. Good overviews of these and other applications are given in~\cite{paige1994history,bai1992csd}.

\paragraph{Organization}
This paper is organized as follows.  After introducing the generalized CS decomposition for partial isometries in Section~\ref{sec:prelim}, we detail in Section~\ref{sec:algorithm} our proposed algorithm for computing it via two polar decompositions and a symmetric eigendecomposition.  We prove backward stability of the algorithm in Section~\ref{sec:backwardstability} under mild hypotheses on the eigensolver and the algorithm used to compute the polar decomposition.  In Section~\ref{sec:polareig}, we highlight a specific pair of algorithms -- Zolo-pd and Zolo-eig -- for computing the polar decomposition and symmetric eigendecomposition.  
The resulting algorithm for the CS decomposition -- Zolo-csd -- is tested in Section~\ref{sec:numerical} on several numerical examples.

\section{Preliminaries} \label{sec:prelim}

In this section, we introduce a generalization of the CS decomposition that is applicable to partial isometries.  We then discuss a few issues concerning partial isometries in finite precision arithmetic.

To begin, recall that every matrix $A \in \mathbb{C}^{m \times n}$ admits a unique \emph{canonical polar decomposition} $A=UH$, where $U \in \mathbb{C}^{m \times n}$ is a partial isometry, $H$ is Hermitian positive semidefinite, and $\mathrm{range}(U^*)=\mathrm{range}(H)$~\cite[Theorem 8.3]{higham2008functions}.  If $A=P\Sigma Q^*$ is the compact singular value decomposition of $A$, i.e. $P \in \mathbb{C}^{m \times r}$ and $Q \in \mathbb{C}^{n \times r}$ have orthonormal columns and $\Sigma \in \mathbb{C}^{r \times r}$ is diagonal with positive diagonal entries (where $r=\mathrm{rank}(A)$), then $U = PQ^*$ and $H=Q\Sigma Q^* = (A^*A)^{1/2}$.

If $m \ge n$, then $A \in \mathbb{C}^{m \times n}$ also admits a \emph{polar decomposition} $A=WH$, where $W \in \mathbb{C}^{m \times n}$ has orthonormal columns and $H$ is Hermitian positive semidefinite.  In general, the canonical polar decomposition $A=UH$ differs from the polar decomposition $A=WH$, which is only defined for $m \ge n$.  When $m \ge n$, the two $H$'s coincide, and if $A$ has rank $n$, then $W$ is uniquely determined, $W=U$, and $H$ is positive definite.

We make use of both decompositions in this paper; the latter is used in the following theorem.


\begin{theorem} \label{thm:csdpartial}
Let $A \in \mathbb{C}^{m \times n}$ ($m \ge 2n$) be a partial isometry of rank $r$. For any partition
\[
A = \begin{pmatrix} A_1 \\ A_2 \end{pmatrix}, \quad A_1 \in \mathbb{C}^{m_1 \times n}, \, A_2 \in \mathbb{C}^{m_2 \times n}, \, m_1,m_2 \ge n, \, m_1 + m_2 = m,
\]
there exist $U_1 \in \mathbb{C}^{m_1 \times n}$, $U_2 \in \mathbb{C}^{m_2 \times n}$, and $C,S,V_1 \in \mathbb{C}^{n \times n}$ such that $U_1$, $U_2$, and $V_1$ have orthonormal columns, $C$ and $S$ are diagonal with nonnegative entries, $C^2+S^2 = \left(\begin{smallmatrix}I_r & 0 \\ 0 & 0_{n-r} \end{smallmatrix}\right)$, and
\begin{equation}
  \label{eq:partCSdef}
A = \begin{pmatrix} U_1 & 0 \\ 0 & U_2 \end{pmatrix} \begin{pmatrix} C \\ S \end{pmatrix} V_1^*.  
\end{equation}
\end{theorem}
\begin{proof}
For $i=1,2$, let $A_i = W_i H_i$ be a polar decomposition of $A_i$.  Observe that since $A_1 A^*A = A_1$,
\begin{align*}
A_1 A_2^* A_2 
&= A_1 (A^*A - A_1^* A_1) \\
&= A_1 - A_1 A_1^* A_1.
\end{align*}
Thus,
\begin{align*}
H_1^2 H_2^2
&= (A_1^* A_1) (A_2^* A_2) \\
&= A_1^* (A_1 A_2^* A_2) \\
&= A_1^*A_1 - (A_1^* A_1)^2
\end{align*}
and
\begin{align*}
H_2^2 H_1^2
&= (A_2^* A_2) (A_1^* A_1) \\
&= (A_1 A_2^* A_2)^* A_1 \\
&= A_1^* A_1 - (A_1^* A_1)^2.
\end{align*}
This shows that $H_1^2$ and $H_2^2$ are commuting Hermitian positive semidefinite matrices, so they are simultaneously diagonalizable~\cite[Section 8.7.2]{golub2012matrix}: $H_1^2 = V \Lambda_1 V^*$ and $H_2^2 = V \Lambda_2 V^*$ for some unitary $V \in \mathbb{C}^{n \times n}$ and diagonal $\Lambda_1, \Lambda_2 \in \mathbb{C}^{n \times n}$ with nonnegative entries.  Moreover,
\begin{align*}
V (\Lambda_1 + \Lambda_2) V^*
&= H_1^2 + H_2^2 \\
&= A_1^* A_1 + A_2^* A_2 \\
&= A^*A.
\end{align*}
The matrix $\Lambda_1+\Lambda_2$, being similar to $A^*A$, has $r$ eigenvalues equal to 1 and $n-r$ equal to 0.  Since it is diagonal, we may order the columns of $V$ so that 
\[
\Lambda_1 + \Lambda_2 = \begin{pmatrix} I_r & 0 \\ 0 & 0_{n-r} \end{pmatrix}.
\]
The theorem follows from taking $U_1 = W_1 V$, $U_2 = W_2 V$, $C = \Lambda_1^{1/2}$, $S = \Lambda_2^{1/2}$, and $V_1 = V$.
\end{proof}

Note that~\eqref{eq:partCSdef} also proves the existence of an ``economical'' rank-deficient CS decomposition
  \begin{equation}
    \label{eq:partCSecon}
A = \begin{pmatrix} U_{1r} & 0 \\ 0 & U_{2r} \end{pmatrix} \begin{pmatrix} C_r \\ S_r \end{pmatrix} V_{1r}^*,     
  \end{equation}
where the subscript $r$ indicates the submatrices consisting of the leading $r$ columns (and rows for $C_r,S_r$): $U_{1r} \in \mathbb{C}^{m_1 \times r}$, $U_{2r} \in \mathbb{C}^{m_2 \times r}$, $C_r,S_r \in \mathbb{C}^{r \times r},$ and 
$V_{1r} \in \mathbb{C}^{n \times r}$. 

\subsection{Approximate partial isometries} \label{sec:partialiso}

In finite precision arithmetic, we will be interested in computing the CS decomposition of matrices that are approximate partial isometries, in the sense that $\|AA^*A-A\|$ is small.
The next pair of lemmas show that if $A$ is a matrix for which $\|AA^*A-A\|$ is small, then $A$ is close to a partial isometry $U$ whose rank coincides with the numerical rank of $A$.

We begin with a few definitions.  For a given unitarily invariant norm $\|\cdot\|$ and a given number $\varepsilon>0$, we define the $\varepsilon$-rank of a matrix $A \in \mathbb{C}^{m \times n}$ to be
\begin{equation} \label{rankeps}
\mathrm{rank}_\varepsilon(A) = \min_{\substack{B \in \mathbb{C}^{m \times n} \\ \|A-B\| \le \varepsilon}} \mathrm{rank}(B).
\end{equation}
Note that if $A = \sum_{k=1}^{\min(m,n)} \sigma_k u_k v_k^*$ is the singular value decomposition of $A$ in summation form, then the minimizer of~(\ref{rankeps}) is given by
\begin{equation} \label{truncatedSVD}
A_r = \sum_{k=1}^r \sigma_k u_k v_k^*,
\end{equation}
where $r = \mathrm{rank}_\varepsilon(A)$.

Let
\begin{equation} \label{distpartialiso}
d(A) =  \min_{\substack{U \in \mathbb{C}^{m \times n} \\ UU^*U = U}} \|A-U\|.
\end{equation}
It can be shown~(\cite{laszkiewicz2006approximation}) in the spectral norm that
\begin{equation}  \label{eq:dAeqn}
d(A) = \max_{1 \le i \le \min(m,n)} \min(\sigma_i(A),|1-\sigma_i(A)|),  
\end{equation}
and this minimum is achieved by the factor $U$ in the canonical polar decomposition $A_r=UH$, where $A_r$ is given by~(\ref{truncatedSVD}) and $r$ is the largest integer such that $\sigma_r \ge 1/2$.

Throughout this paper, we make use of the fact that in any unitarily invariant norm,
\begin{equation} \label{normABC}
\|ABC\| \le \min\{\sigma_1(A)\sigma_1(B)\|C\|, \, \sigma_1(A)\|B\|\sigma_1(C), \, \|A\|\sigma_1(B)\sigma_1(C) \}
\end{equation}
for any matrices $A$, $B$, $C$ whose product $ABC$ is defined~\cite[Equation (B.7)]{higham2008functions}.

The following lemma extends~\cite[Lemma 8.17]{higham2008functions} to approximate partial isometries having exact rank $r$.

\begin{lemma} \label{lemma:partialiso}
Let $A \in \mathbb{C}^{m \times n}$ have canonical polar decomposition $A=UH$.  If $\mathrm{rank}(A)=r$, then
\begin{equation}
\frac{\|AA^*A-A\|}{\sigma_1(A)(1+\sigma_1(A))} \le \|A-U\| \le \frac{\|AA^*A-A\|}{\sigma_r(A)(1+\sigma_r(A))}
\end{equation}
in any unitarily invariant norm.
\end{lemma}
\begin{proof}
Let $A=P\Sigma Q^*$ be the compact singular value decomposition of $A$ with $\Sigma\in\mathbb{R}^{r\times r}$, so that $U=PQ^*$.  Then
\begin{align*}
\|AA^*A-A\| 
&= \|\Sigma^3-\Sigma\| \\
&= \|\Sigma(\Sigma-I)(\Sigma+I)\| \\
&\le \sigma_1(A) \|\Sigma-I\| (\sigma_1(A)+1) \\
&= \sigma_1(A) \|A-U\| (\sigma_1(A)+1).
\end{align*}
On the other hand,
\begin{align*}
\|A-U\| 
&= \|\Sigma-I\| \\
&= \|\Sigma^{-1} (\Sigma^3-\Sigma) (\Sigma+I)^{-1}\| \\
&\le \sigma_1(\Sigma^{-1}) \|\Sigma^3-\Sigma\| \sigma_1( (\Sigma+I)^{-1} ) \\
&= \frac{1}{\sigma_r(A)} \|AA^*A-A\| \, \frac{1}{\sigma_r(A)+1}.
\end{align*}
\end{proof}

The next lemma handles the setting in which $\|AA^*A-A\|$ is small and $A$ has $\varepsilon$-rank $r$ rather than exact rank $r$.

\begin{lemma} \label{lemma:nearbypartialiso}
Let $A \in \mathbb{C}^{m \times n}$ have $\varepsilon$-rank $r$ with respect to a unitarily invariant norm $\|\cdot\|$.  Then there exists a partial isometry $U \in \mathbb{C}^{m \times n}$ of rank $r$ satisfying
\begin{equation} \label{nearbypartialiso}
\|A-U\| \le \varepsilon + \frac{\|AA^*A-A\|+\varepsilon(1+3\sigma_1(A)^2)}{\sigma_r(A)(1+\sigma_r(A))}.
\end{equation}
\end{lemma}
\begin{proof}
Let $A_r$ be as in~(\ref{truncatedSVD}), and let $A_r=UH$ be the canonical polar decomposition of $A_r$.  
By Lemma~\ref{lemma:partialiso},
\begin{align*}
\|A-U\| \le \|A-A_r\| + \frac{\|A_r A_r^* A_r - A_r\|}{\sigma_r(A_r)(1+\sigma_r(A_r))}.
\end{align*}
Now since $\sigma_1(A_r)=\sigma_1(A)$, the inequality~(\ref{normABC}) implies
\begin{equation*}\begin{split}
\|A_r A_r^* A_r - A_r\|
&\le \|AA^*A-A\| + \|A_r-A\| + \|A_r A_r^* A_r - AA^*A\| \\
&\le \|AA^*A-A\| + \|A_r-A\| + \|(A_r-A)A_r^*A_r\| + \|A(A_r-A)^*A_r\| \\&\quad + \|AA^*(A_r-A)\| \\
&\le \|AA^*A-A\| + (1+3\sigma_1(A)^2) \|A_r-A\|.
\end{split}\end{equation*}
The result then follows from the relations $\sigma_r(A_r) = \sigma_r(A)$ and $\|A-A_r\| \le \varepsilon$.
\end{proof}

Note that Lemmas~\ref{lemma:partialiso} and \ref{lemma:nearbypartialiso} can be used to estimate $d(A)$ in~\eqref{distpartialiso}.

\section{Algorithm} \label{sec:algorithm}

In this section, we detail a general algorithm for computing the CS decomposition via two polar decompositions and a symmetric eigendecomposition.  We focus on the motivation behind the algorithm first, postponing an analysis of its stability to Section~\ref{sec:backwardstability}.

The observations made in Section~\ref{sec:intro} immediately suggest the following general strategy for computing the CS decomposition of a matrix $A = \left( \begin{smallmatrix} A_1 \\ A_2 \end{smallmatrix} \right)$ with $A_1,A_2 \in \mathbb{C}^{n \times n}$ and $A^*A=I$.  One can compute polar decompositions $A_1=W_1 H_1$ and $A_2 = W_2 H_2$, diagonalize $H_1$ and/or $H_2$ to obtain $V_1$, $C$ and $S$ such that $H_1 = V_1 C V_1^*$ and $H_2 = V_1 S V_1^*$, and set $U_1 = W_1 V_1$, $U_2 = W_2 V_1$.  For such an approach to be viable, it is critical that $V_1$ be computed in a stable way.  The following MATLAB example illustrates the pitfalls of a naive strategy: diagonalizing $H_1$.
\begin{lstlisting}[style=Matlab-editor]
theta = [1e-8 2e-8 3e-8];
C = diag(cos(theta));
S = diag(sin(theta));
V1 = [2 -1 2; 2 2 -1; 1 -2 -2]/3;
H1 = V1*C*V1';
H2 = V1*S*V1';
[V1,C] = eig(H1);
S = V1'*H2*V1;
max(max(abs(S-diag(diag(S)))))
     ans = 2.8187e-09
\end{lstlisting}
%
Here, the computed $V_1$ (which we will denote by $\widehat{V}_1$) provides a poor approximation of the eigenvectors of $H_2$, as evidenced by the size of the off-diagonal entry of $\widehat{V}_1^* H_2 \widehat{V}_1$ with the largest absolute value.  Ideally, in double-precision arithmetic, the latter quantity should be a small multiple of the unit roundoff $u=2^{-53} \approx 10^{-16}$.

One possible remedy is to perform simultaneous diagonalization~\cite{bunse1993numerical} to obtain $V_1$ from $H_1$ and $H_2$.  In this paper, we consider a different approach that exploits the special structure of $H_1$ and $H_2$.  The idea is to obtain $V_1$ by computing the eigendecomposition of $H_2-H_1$, whose eigenvectors are the same as those of $H_1$ and $H_2$.  The advantages of this approach, though not obvious at first glance, are easily illustrated.  In the MATLAB example above, replacing \lstinline[style=Matlab-editor]!eig(H1)! with \lstinline[style=Matlab-editor]!eig(H2-H1)! yields \lstinline[style=Matlab-editor]!ans = 1.7806e-17!.

To give more insight, we explain what goes wrong if we obtain $V_1$ via an eigendecomposition of $H_1$. 
Since $\cos\theta_i \approx 1-\tfrac{1}{2}\theta_i^2$ for small $\theta_i$, we have $|\cos\theta_i-\cos\theta_j| \approx \tfrac{1}{2}(\theta_i+\theta_j) |\theta_i-\theta_j| \ll |\theta_i-\theta_j|$ for small $\theta_i$ and $\theta_j$, rendering the eigenvectors of $H_1$ very sensitive to perturbation if two or more angles are close to zero.
A standard eigendecomposition algorithm for $H_1$ still gives a backward stable decomposition: $\|\widehat{V}_1^* H_1 \widehat{V}_1 - \widehat{\Lambda}_1\| \le \epsilon \|H_1\|$, where $\widehat{\Lambda}_1$ is diagonal and $\epsilon$ is a small multiple of $u$. However, inaccuracies in the columns of $\widehat{V}_1$ manifest themselves when we use the same $\widehat{V}_1$ to compute the eigendecomposition of $H_2$: 
For a computed eigenpair $(\widehat\lambda_i,\widehat v_i)$ 
with $\|\widehat v_i\|_2=1$
obtained in a backward stable manner, we have 
$H_1\widehat v_i=\widehat \lambda_i \widehat v_i+\epsilon$ where $\|\epsilon\|=O(u)$, where $u$ is the unit roundoff.  Expanding $\widehat v_i=\sum_{j=1}^nc_jv_j$ where $v_j$ are the exact eigenvectors of $H_1$, we have
$|v_j^* \epsilon| =  |c_j||\widehat\lambda_i-\lambda_j|$.
Now, the same $\widehat v_i$ taken as an approximate eigenvector of $H_2$ gives 
$H_2\widehat v_i=\widehat \lambda_{i,2} \widehat v_i+\epsilon_2$, where 
the choice $\widehat \lambda_{i,2}=\widehat v_i^*H_2\widehat v_i$ minimizes $\|\epsilon_2\|$.  We then have $|v_j^*\epsilon_2| = |\widehat \lambda_{i,2}-\lambda_{j,2}||c_j|$ for each $j$.  Using the above relation $|c_j|=\frac{|v_j^* \epsilon|}{|\widehat\lambda_i-\lambda_j|}$, we see that $|v_j^*\epsilon_2| =
|v_j^*\epsilon| \frac{|\widehat \lambda_{i,2}-\lambda_{j,2}|}{|\widehat\lambda_i-\lambda_j|}$ for each $j$. The crucial observation is that for each $j$, the $v_j$-component of $\epsilon$ is magnified by the factor 
$\frac{|\widehat \lambda_{i,2}-\lambda_{j,2}|}{|\widehat\lambda_i-\lambda_j|}$, the ratio in the eigenvalue gap.
In the above setting, $\lambda_{i}=\cos\theta_i$ and $\lambda_{i,2}=\sin\theta_i$, and since the eigenvalues have $O(\|\epsilon\|^2)$ accuracy~\cite{parlettsym}, no essence is lost in taking $\widehat\lambda_{i}=\lambda_{i}$ and $\widehat\lambda_{i,2}=\lambda_{i,2}$.  Thus, since  
$|\sin\theta_i-\sin\theta_j| \approx |\theta_i-\theta_j|$ and $|\cos\theta_i-\cos\theta_j| \approx \tfrac{1}{2}(\theta_i+\theta_j) |\theta_i-\theta_j|$ for small $\theta_i$ and $\theta_j$, the relative gap $|\sin\theta_i-\sin\theta_j| / |\cos\theta_i-\cos\theta_j| \approx 2/(\theta_i+\theta_j)$ can be arbitrarily large, in which case $\widehat{V}_2$ does not give a backward stable eigendecomposition for $H_2$: $\|\widehat{V}_1^* H_2 \widehat{V}_1 - \widehat{\Lambda}_2\| \gg  u\|H_2\|$.

A similar problem occurs if $V_1$ is obtained via an eigendecomposition of $H_2$.  If two or more angles are close to $\tfrac{\pi}{2}$, their sines are closely spaced, rendering the eigenvectors of $H_2$ very sensitive to perturbation.  In this scenario, numerical experiments show that $\widehat{V}_1^* H_1 \widehat{V}_1$ can have off-diagonal entries with unacceptably large magnitude.  The essence of the problem is that for $\theta_i \neq \theta_j$ near $\tfrac{\pi}{2}$, the ratio $|\cos\theta_i-\cos\theta_j| / |\sin\theta_i-\sin\theta_j| \approx 2/(\pi-\theta_i-\theta_j)$ can be arbitrarily large.

Obtaining $V_1$ via an eigendecomposition of $H_2-H_1$ sidesteps these difficulties for the following reason.  The function $g(\theta)=\sin\theta-\cos\theta$ has derivative $g'(\theta) \ge 1$ on $[0,\tfrac{\pi}{2}]$, from which it is easy to show that $|\cos\theta_i-\cos\theta_j|/|g(\theta_i)-g(\theta_j)| \le 1$ and $|\sin\theta_i-\sin\theta_j|/|g(\theta_i)-g(\theta_j)| \le 1$ for every $\theta_i,\theta_j \in [0,\tfrac{\pi}{2}]$ with $\theta_i \neq \theta_j$.  In other words, the eigenvalues of $H_1$ and $H_2$ are spaced no further apart than the eigenvalues of $H_2-H_1$.  As a result, the arguments in the preceding paragraphs suggest that the numerically computed eigenvectors of $H_2-H_1$ likely provide a backward stable approximation of the eigenvectors of both $H_1$ and $H_2$.  As an aside, note that another seemingly natural alternative -- computing the eigendecomposition of $H_1+H_2$ -- is not viable since the derivative of $\cos\theta+\sin\theta$ vanishes at $\theta=\tfrac{\pi}{4}$.

\paragraph{Extension to partial isometries} 
With one caveat, all of the arguments in the preceding paragraph carry over to the more general setting in which $A = \left(\begin{smallmatrix}A_1 \\ A_2\end{smallmatrix}\right)$ is a partial isometry with $A_1 \in \mathbb{C}^{m_1 \times n}$, $A_2 \in \mathbb{C}^{m_2 \times n}$, and $m_1,m_2 \ge n$.  The caveat is that if $A$ is rank-deficient and has principal angle(s) $\theta_i$ equal to $\pi/4$, then it may be impossible to distinguish between two eigenspaces of $H_1$ and $H_2$: the eigenspace $\mathcal{V}_0$ corresponding to the eigenvalue 0, and the eigenspace $\mathcal{V}_{1/\sqrt{2}}$ corresponding to the eigenvalue $\cos(\pi/4)=\sin(\pi/4)=1/\sqrt{2}$.   Indeed, both of these eigenspaces correspond to the zero eigenvalue of $H_2-H_1$.  Even if $\theta_i \neq \pi/4$ for every $i$, numerical instabilities can still arise if any angle $\theta_i$ is close to $\pi/4$.

Fortunately, there is a simple remedy to this problem. When $A$ is rank-deficient, then instead of computing the eigendecomposition of $H_2-H_1$, one can compute the eigendecomposition of $B=H_2-H_1 + \mu(I-A^*A)$, where $\mu > 1$ is a scalar.  This has the effect of shifting the eigenvalue corresponding to $\mathcal{V}_0$ away from all of the other eigenvalues of $B$.  Indeed, if $H_1 = V_1 \left(\begin{smallmatrix} C_r & 0 \\ 0 & 0_{n-r}\end{smallmatrix}\right) V_1^*$ and $H_2 = V_1 \left(\begin{smallmatrix} S_r & 0 \\ 0 & 0_{n-r}\end{smallmatrix}\right) V_1^*$, then
\[
B = V_1 \begin{pmatrix} S_r-C_r & 0 \\ 0 & \mu I_{n-r} \end{pmatrix} V_1^*,
\]
and the diagonal entries of $S_r-C_r$ lie in the interval $[-1,1] \not\owns \mu$.

\paragraph{Algorithm summary} The algorithm that results from these considerations is summarized below.  In what follows, we use $\mathrm{diag}$ to denote (as in MATLAB) the operator that, if applied to a matrix $X \in \mathbb{C}^{n \times n}$, returns a vector $x \in \mathbb{C}^n$ with $x_i=X_{ii}$, and, if applied to a vector $x \in \mathbb{C}^n$, returns $X \in \mathbb{C}^{n \times n}$ with $X_{ii}=x_i$ and $X_{ij}=0$ for $i \neq j$.

\begin{breakablealgorithm} \label{alg:csd}
\caption[CS decomposition]{CS decomposition of a partial isometry $A = \left(\begin{smallmatrix}A_1 \\ A_2\end{smallmatrix}\right)$, $A_1 \in \mathbb{C}^{m_1 \times n}$, $A_2 \in \mathbb{C}^{m_2 \times n}$, $m_1,m_2 \ge n$}
\begin{algorithmic}[1]
\STATE{$W_1 H_1 = A_1$ (polar decomposition)} \label{line:polar1}
\STATE{$W_2 H_2 = A_2$ (polar decomposition)} \label{line:polar2}
\IFTHENELSE{$\mathrm{rank}(A)=n$}{$\mu=0$}{$\mu=2$} \label{line:mu}
\STATE{$B = H_2-H_1 + \mu(I-A^*A)$} \label{line:B}
\STATE{$V_1 \Lambda V_1^* = B$ (symmetric eigendecomposition)} \label{line:eig}
\STATE{$U_1 = W_1 V_1$} \label{line:U1}
\STATE{$U_2 = W_2 V_1$} \label{line:U2}
\STATE{$C = \mathrm{diag}(\mathrm{diag}(V_1^* H_1 V_1))$} \label{line:C}
\STATE{$S = \mathrm{diag}(\mathrm{diag}(V_1^* H_2 V_1))$} \label{line:S}
\RETURN $U_1,U_2,C,S,V_1$
\end{algorithmic}
\end{breakablealgorithm}\vspace{0.05in}

We conclude this section with a few remarks.

\begin{enumerate}[label=(\thesection.\roman*),ref=\thesection.\roman*]
\item The algorithm treats $A_1$ and $A_2$ in a symmetric way, in the sense that when $\mu=0$, exchanging the roles of $A_1$ and $A_2$ merely negates $B$ and $\Lambda$ (thereby sending $\theta_i = \arctan(S_{ii}/C_{ii})$ to $\pi/4-\theta_i$ for each $i$).  
\item \label{parallel} Lines~\ref{line:polar1}-\ref{line:polar2} and lines~\ref{line:U1}-\ref{line:S} can each be carried out in parallel.  Furthermore, if $A^*A$ is needed in Line~\ref{line:B}, then it can be computed in parallel with lines~\ref{line:polar1}-\ref{line:polar2}.
\item \label{postprocessing} As a post-processing step, one can compute $\theta_i = \arctan(S_{ii}/C_{ii})$, $i=1,2,\dots,n$, and overwrite $C$ and $S$ with $\mathrm{diag}(\cos\theta)$ and $\mathrm{diag}(\sin\theta)$, respectively (with the obvious modifications for rank-deficient $A$).  
It is not hard to verify that this has the effect of reducing $\left\|C^2+S^2-\left(\begin{smallmatrix} I_r & 0 \\ 0 & 0_{n-r} \end{smallmatrix}\right)\right\|$ without disrupting the backward stability of the algorithm.
\item An alternative, cheaper way to compute $C$ and $S$ in lines~\ref{line:C}-\ref{line:S} is to solve the equation $\sin\theta_i-\cos\theta_i = \Lambda_{ii}$ for $\theta_i$, $i=1,2,\dots,n$, and set $C=\mathrm{diag}(\cos\theta)$, $S=\mathrm{diag}(\sin\theta)$ (with the obvious modifications for rank-deficient $A$).  Our numerical experiments suggest that this approach is generally less accurate than lines~\ref{line:C}-\ref{line:S}, but it still renders the algorithm backward stable.  Our analysis will focus on the use of lines~\ref{line:C}-\ref{line:S} to obtain $C$ and $S$, but it can be easily modified to treat the alternative approach.
\item Let us examine the arithmetic cost in flop counts. 
The steps that require $O(n^3)$ flops 
are two polar decompositions (lines~\ref{line:polar1}--\ref{line:polar2})
and a symmetric eigendecomposition (line~\ref{line:eig}), in addition to matrix-product operations whose flop counts are clear: $A^*A$ (line~\ref{line:B}, costing $(m_1+m_2)n^2$ flops exploiting symmetry) and $U_1,U_2$ (lines~\ref{line:U1}--\ref{line:U2}, $2 m_i n^2$ flops for each $i=1,2$), and the diagonal elements of $C$ and $S$ (lines~\ref{line:C}--\ref{line:S}, $2n^3$ flops each). The costs of the polar and eigenvalue decompositions depend on the algorithm used. When Zolo-pd and Zolo-eig are used, they are $64 m_i n^2+\frac{8}{3}n^3$ flops for each $i=1,2$ ($8\max\{m_1,m_2\}n^2+\frac{1}{3}n^3$ along the critical path) for Zolo-pd and about $55n^3$ flops  ($16n^3$) for Zolo-eig~\cite[Table~5.1,5.2]{nakatsukasa2016computing}. Zolo-csd thus requires a total of about $67(m_1+m_2)n^2+64n^3$ flops ($10\max\{m_1,m_2\}n^2+16n^3$ along the critical path). Clearly, the polar and eigenvalue decompositions form the majority of the computation. When a classical algorithm is used for these decompositions (via the SVD for polar, costing $8 m_i n^2+20n^3$ and $9n^3$ for the eigendecomposition), the overall cost is $11(m_1+m_2)n^2+53n^3$ flops ($10\max\{m_1,m_2\}n^2+29n^3$ along the critical path). 
It is worth noting that these flop counts usually do not accurately reflect the actual running time, particularly in a massively parallel computing environment; they are presented here for reference purposes.

\item In applications, we expect that users will know in advance whether $A$ is full-rank or not. In the rare situation in which it is not known until runtime, 
an inexpensive approach is to compute $\|A\|_F$, noting that assuming $d(A)\ll 1$, we have $\|A\|_F\approx \sqrt{\mbox{rank}(A)}$. Another alternative is to perform a rank-revealing $QR$ factorization (for instance) in line~\ref{line:mu}.  Needless to say, one should measure the $\varepsilon$-rank of $A$ for a suitable tolerance $\varepsilon>0$, not the exact rank of $A$.
\item In the rank-deficient case $r<n$, an economical CS decomposition~\eqref{eq:partCSecon} can be obtained by a simple modification as follows:
After line~\ref{line:eig}, we extract the eigenvalues $\Lambda_{ii}\in [-1,1]$ 
(there should be $r$ of them)
and their corresponding eigenvectors $V_{1r}\in\mathbb{C}^{n\times r}$ (the remaining $n-r$ columns of $V_1$ are the null vectors of $A$). 
We have $H_2-H_1= V_{1r}\Lambda_rV_{1r}^*$, where $\Lambda_r \in \mathbb{C}^{r \times r}$ is diagonal with the $r$ eigenvalues of $\Lambda$ lying in $[-1,1]$ on its diagonal (assuming $A$ is an exact partial isometry). 
Finally, we let $U_i:=W_iV_{1r}\in \mathbb{C}^{m_i \times r}$, 
and $C = \mathrm{diag}(\mathrm{diag}(V_{1r}^* H_1 V_{1r}))\in \mathbb{C}^{r \times r}$, 
$S = \mathrm{diag}(\mathrm{diag}(V_{1r}^* H_2 V_{1r}))\in \mathbb{C}^{r \times r}$
to obtain the rank-deficient CS decomposition $A_1=U_1CV_{1r}^*,A_2=U_2SV_{1r}^*$ (we output $V_1:=V_{1r}$). 
\item \label{remark:2x2} It is well-known that any unitary $A = \left( \begin{smallmatrix} A_1 & A_3 \\ A_2 & A_4 \end{smallmatrix} \right) \in \mathbb{C}^{2n \times 2n}$ admits a \emph{complete $2 \times 2$ CS decomposition}
\[
A = \begin{pmatrix} U_1 & 0 \\ 0 & U_2 \end{pmatrix} \begin{pmatrix} C & -S \\ S & C \end{pmatrix} \begin{pmatrix} V_1 & 0 \\ 0 & V_2 \end{pmatrix}^*,
\]
where $U_1,U_2,V_1,V_2 \in \mathbb{C}^{n \times n}$ are unitary, $C,S \in \mathbb{C}^{n \times n}$ are diagonal with nonnegative entries, and $C^2+S^2=I$~\cite[Section 2.6.4]{golub2012matrix}.  Algorithm~\ref{alg:csd} applied to $\left( \begin{smallmatrix} A_1 \\ A_2 \end{smallmatrix} \right)$ computes all of these matrices except $V_2$.  If $V_2$ is desired, we advocate using the following strategy from~\cite[Section 4.11]{sutton2013divide}: 
compute $X = -A_3^* U_1 S + A_4^* U_2 C$ and its 
QR decomposition $X=QR$, and set $V_2 = Q$ (the QR decomposition may be unnecessary, because if $A$ is exactly unitary, then so is $X$). 
An argument similar to the proof of~\cite[Theorem 18]{sutton2013divide} shows that this algorithm for the $2 \times 2$ CS decomposition is backward stable if Algorithm~\ref{alg:csd} is backward stable.
\end{enumerate}

\section{Backward stability} \label{sec:backwardstability}

In this section, we prove that Algorithm~\ref{alg:csd} is backward stable, provided that the polar decompositions in lines~\ref{line:polar1}-\ref{line:polar2} and the eigendecomposition in line~\ref{line:eig} are computed in a backward stable way.  

Throughout this section, we continue to use $\|\cdot\|$ to denote any unitarily invariant norm.  We denote by $c_n$ the absolute condition number of the map $\mathcal{H}$ sending a matrix $A \in \mathbb{C}^{m \times n}$ ($m \ge n$) to the Hermitian positive semidefinite factor $H=\mathcal{H}(A)$ in the polar decomposition $A=WH$; that is, 
\begin{equation} \label{condH}
c_n = \max_{\substack{A,\Delta A \in \mathbb{C}^{m \times n} \\ \Delta A \neq 0}} \frac{\|\mathcal{H}(A+\Delta A)-\mathcal{H}(A)\|}{\|\Delta A\|}.
\end{equation}
It is easy to check that this number depends on $n$ and $\|\cdot\|$ but not $m$.  In the Frobenius norm, $c_n=\sqrt{2}$ is constant~\cite[Theorem 8.9]{higham2008functions}. In the 2-norm, it is known that $b_n \le c_n \le 1 + 2b_n$ where $b_n \sim \frac{2}{\pi}\log n$.~\cite[Corollary 4.3]{mathias1993hadamard}.  We will also make use of the fact that
\[
\|\mathrm{diag}(\mathrm{diag}(A))\| \le \|A\|
\]
in any unitarily invariant norm~\cite[p. 152]{bhatia1989comparing}.

\begin{theorem} \label{thm:backwardstability}
Let $A = \left(\begin{smallmatrix} A_1 \\ A_2\end{smallmatrix}\right) \in \mathbb{C}^{m \times n}$ with $A_1 \in \mathbb{C}^{m_1 \times n}$, $A_2 \in \mathbb{C}^{m_2 \times n}$, $m_1,m_2 \ge n$, and $m=m_1+m_2$.  
Suppose that Algorithm~\ref{alg:csd} computes the following quantities with indicated errors:
\begin{align}
\widehat{W}_1 \widehat{H}_1 &= A_1 + \Delta A_1, \label{polarA1} \\
\widehat{W}_2 \widehat{H}_2 &= A_2 + \Delta A_2, \label{polarA2} \\
\widehat{B} &= \widehat{H}_2 - \widehat{H}_1 + \mu (I-A^*A) + \Delta B_1, \\
\widehat{V}_1 \widehat{\Lambda} \widehat{V}_1^* &= \widehat{B} + \Delta B_2, \\
\widehat{C} &= \mathrm{diag}(\mathrm{diag}(\widehat{V}_1^* \widehat{H}_1 \widehat{V}_1)) + \Delta C, \\
\widehat{S} &= \mathrm{diag}(\mathrm{diag}(\widehat{V}_1^* \widehat{H}_2 \widehat{V}_1)) + \Delta S,
\end{align}
where $\widehat{\Lambda}$, $\widehat{C}$, and $\widehat{S}$ are real and diagonal, $\widehat{H}_1$ and $\widehat{H}_2$ are Hermitian, and $\mu \ge 0$. 
Assume that $\|\widehat{W}_1^*\widehat{W}_1-I\|$, $\|\widehat{W}_2^*\widehat{W}_2-I\|$, $\|\widehat{V}_1^*\widehat{V}_1-I\|$, $\min_{G=G^* \ge 0}\|\widehat{H}_1-G\|$, and $\min_{G=G^* \ge 0}\|\widehat{H}_2-G\|$ are each bounded above by a number $\delta$. 
If~(\ref{distpartialiso}) has no minimizer of rank $n$, assume further that $\mu>1$.
Then
\begin{align}
\| \widehat{W}_1 \widehat{V}_1 \widehat{C} \widehat{V}_1^* - A_1 \| &\le (4c_n+1) \|\Delta A_1\| + 2c_n\|\Delta A_2\| + 2\|\Delta B_1\| + 2\|\Delta B_2\| + \|\Delta C\| \nonumber\\&\quad + (9c_n+10+2\max\{\mu,1\}) \delta + (6c_n+4\mu) d(A) + o(\eta), \label{backwarderror1} \\
\| \widehat{W}_2 \widehat{V}_1 \widehat{S} \widehat{V}_1^* - A_2 \| &\le 2c_n \|\Delta A_1\| + (4c_n+1)\|\Delta A_2\| + 2\|\Delta B_1\| + 2\|\Delta B_2\| + \|\Delta S\| \nonumber\\&\quad +  (9c_n+10+2\max\{\mu,1\}) \delta + (6c_n+4\mu) d(A) + o(\eta), \label{backwarderror2}
\end{align}
asymptotically as 
\begin{equation} \label{maxerr}
\eta := \max\{\delta, d(A), \|\Delta A_1\|, \|\Delta A_2\|, \|\Delta B_1\|, \|\Delta B_2\|, \|\Delta C\|, \|\Delta S\|\} \rightarrow 0.
\end{equation}
\end{theorem}

Before proving the theorem, we make a few remarks.  
First, the smallness of the quantities $\|\Delta A_i\|$, $\|\widehat{W}_i^*\widehat{W}_i-I\|$, and $\min_{G=G^* \ge 0}\|\widehat{H}_i-G\|$ is equivalent to the condition that the polar decompositions $A_i \approx \widehat{W}_i \widehat{H}_i$, $i=1,2$, are computed in a backward stable way~\cite{nakatsukasa2012backward}.  Second, the smallness of $\|\Delta B_2\|$ and $\|\widehat{V}_1^*\widehat{V}_1-I\|$ corresponds to the condition that the eigendecomposition of $\widehat{B}$ is computed in a backward stable way.  Smallness of $\|\Delta B_1\|$, $\|\Delta C\|$, and $\|\Delta S\|$ is automatic in floating point arithmetic. 
We also note that $A$ is not assumed to be exactly a partial isometry; its deviation is measured by $d(A)$.
Theorem~\ref{thm:backwardstability} thus says that Algorithm~\ref{alg:csd} is backward stable whenever 
$d(A)$ is small and
the polar decompositions in lines~\ref{line:polar1}-\ref{line:polar2} and the symmetric eigendecomposition in line~\ref{line:eig} are computed in a backward stable way.  We give examples of backward stable algorithms for the polar decomposition and symmetric eigendecomposition in Section~\ref{sec:polareig}.

The estimates~(\ref{backwarderror1}-\ref{backwarderror2}) have been written in full detail to make clear the contribution of each source of error.  Coarser estimates of a more memorable form are easy to write down. Consider, for example, the setting in which $A$ has nearly orthonormal columns, so that $\mu$ can be taken equal to zero.  Then~(\ref{backwarderror1}-\ref{backwarderror2}) imply that
\begin{align}
\| \widehat{W}_1 \widehat{V}_1 \widehat{C} \widehat{V}_1^* - A_1 \|_2 &\lessapprox \left(39+\frac{84}{\pi}\log n\right)\eta + o(\eta), \label{2normbackwarderror1} \\
\| \widehat{W}_2 \widehat{V}_1 \widehat{S} \widehat{V}_1^* - A_2 \|_2 &\lessapprox \left(39+\frac{84}{\pi}\log n\right)\eta + o(\eta), \label{2normbackwarderror2}
\end{align}
in the 2-norm, and
\begin{align}
\| \widehat{W}_1 \widehat{V}_1 \widehat{C} \widehat{V}_1^* - A_1 
\|_F &\le (18+21\sqrt{2})\eta + o(\eta), \label{Fnormbackwarderror1} \\
\| \widehat{W}_2 \widehat{V}_1 \widehat{S} \widehat{V}_1^* - A_2 \|_F &\le (18+21\sqrt{2})\eta + o(\eta), \label{Fnormbackwarderror2}
\end{align}
in the Frobenius norm, where $\eta$ is given by~(\ref{maxerr}).  (Our numerical experiments suggest that these are pessimistic estimates.)

\paragraph{Proof of Theorem~\ref{thm:backwardstability}}
To prove the theorem, let $\widetilde{A} \in \mathbb{C}^{m \times n}$ be a partial isometry of maximal rank such that $\|A-\widetilde{A}\|=d(A)$.  Let $r = \mathrm{rank}(\widetilde{A})$, and let $\widetilde{A}_1 \in \mathbb{C}^{m_1 \times n}$ and $\widetilde{A}_2 \in \mathbb{C}^{m_2 \times n}$ be such that $\widetilde{A} = \left(\begin{smallmatrix} \widetilde{A}_1 \\ \widetilde{A}_2\end{smallmatrix}\right)$.
Let
\[
\begin{pmatrix} \widetilde{A}_1 \\ \widetilde{A}_2 \end{pmatrix} = \begin{pmatrix} U_1 & 0 \\ 0 & U_2 \end{pmatrix} \begin{pmatrix} C \\ S \end{pmatrix} V_1^*
\]
be a
CS decomposition of $\widetilde{A}$. Let $0 \le \theta_1 \le \theta_2 \le \cdots \le \theta_r \le \frac{\pi}{2}$ be the corresponding angles such that $C = \left(\begin{smallmatrix} C_r & 0 \\ 0 & 0_{n-r}\end{smallmatrix}\right)$ and $S = \left(\begin{smallmatrix} S_r & 0 \\ 0 & 0_{n-r}\end{smallmatrix}\right)$ with $C_r=\mathrm{diag}(\cos\theta)$ and $S_r=\mathrm{diag}(\sin\theta)$.  Define $W_1 = U_1 V_1^*$, $H_1 = V_1 C V_1^*$, $W_2 = U_2 V_1^*$, and $H_2 = V_1 S V_1^*$, so that $W_i H_i$ is a polar decomposition of $\widetilde{A}_i$ for each $i$.

\begin{lemma}
Let $\Delta H_i = \widehat{H}_i-H_i$, $i=1,2$. Then, to leading order,
\begin{align}
\|\Delta H_i\| &\le \delta + c_n\left(\|A_i - \widetilde{A}_i\| + \|\Delta A_i\| + \tfrac{3}{2}\delta\right), \label{normDeltaH}
\end{align}
where $c_n$ is given by~(\ref{condH}).
\end{lemma}
\begin{proof}
Let $X_i = \argmin_{X^*X=I} \|\widehat{W}_i-X\|$ and $G_i = \argmin_{G=G^* \ge 0} \|\widehat{H}_i-G\|$.  Then
\begin{align*}
X_i G_i 
&= \widehat{W}_i \widehat{H}_i + (X_i - \widehat{W}_i) \widehat{H}_i + X_i (G_i - \widehat{H}_i) \\
&= \widetilde{A}_i + \Delta \widetilde{A}_i,
\end{align*}
where $\Delta \widetilde{A}_i = (A_i - \widetilde{A}_i) + \Delta A_i + (X_i - \widehat{W}_i) \widehat{H}_i + X_i (G_i - \widehat{H}_i)$.
Since $X_i$ is unitary and $G_i$ is Hermitian positive semidefinite, we have
\begin{align*}
\|G_i-H_i\| 
&= \|\mathcal{H}(\widetilde{A}_i+\Delta\widetilde{A}_i) - \mathcal{H}(\widetilde{A}_i)\| \\
&\le c_n (\|A_i - \widetilde{A}_i\| + \|\Delta A_i\| + \|X_i - \widehat{W}_i\| \sigma_1(\widehat{H}_i) + \|G_i - \widehat{H}_i\|).
\end{align*}
By assumption, $\|G_i-\widehat{H}_i\| \le \delta$ and $\|\widehat{W}_i^* \widehat{W}_i - I\| \le \delta$, so $\|X_i-\widehat{W}_i\| \sim \frac{1}{2}\delta$ as $\delta \rightarrow 0$~\cite[Lemma 8.17]{higham2008functions}.  Moreover, $\sigma_1(\widehat{H}_i) \sim \sigma_1(A_i + \Delta A_i) \le \sigma_1(A+(\Delta A_1^*,\Delta A_2^*)^*) \sim 1$.  These facts, together with the inequality
\[
\|\widehat{H}_i - H_i\| \le \|G_i - \widehat{H}_i\| + \|G_i - H_i\|,
\] 
prove~(\ref{normDeltaH}).
\end{proof}

Now let
\[
B=H_2-H_1 + \mu (I-\widetilde{A}^*\widetilde{A}).
\] 
Then 
\[
\widehat{B} - B = \Delta B_1 + \Delta H_2 - \Delta H_1 + \mu (\widetilde{A}-A)^*A + \mu \widetilde{A}^* (\widetilde{A}-A),
\]
so the preceding lemma implies
\begin{equation} \label{Bdiff}
\begin{split}
\|\widehat{B}-B\| \le \|\Delta B_1\| + c_n(\|A_1-\widetilde{A}_1\| + \|A_2-\widetilde{A}_2\| + \|\Delta A_1\| + \|\Delta A_2\|) \\ + (3c_n+2)\delta + \mu(\sigma_1(A)+1)d(A).
\end{split}
\end{equation}

The forthcoming analysis will rely on a certain pair of functions $f$ and $g$ with the property that $f(B)=H_1$, $g(B)=H_2$, $f$ and $g$ have bounded Fr\'echet derivative at $B$, and $f$ and $g$ are analytic on a complex neighborhood of the spectrum of $B$.  Consider first the case in which $r<n$, so that $\mu>1$ (in our algorithm we always take $\mu=2$ when $r<n$).
Let 
\[
\rho = 
\begin{cases}
\frac{1+\mu}{2} &\mbox{ if } 1 < \mu < 2\sqrt{2}-1, \\
\sqrt{2} &\mbox{ otherwise, }
\end{cases}
\]
and define
\[
f(z) =
\begin{cases}
\frac{1}{2}(-z + \sqrt{2-z^2}), &\mbox{ if } |z|<\rho, \\
0, &\mbox{ if } |z| \ge \rho,
\end{cases}
\]
and
\[
g(z) =
\begin{cases}
\frac{1}{2}(z + \sqrt{2-z^2}), &\mbox{ if } |z|<\rho, \\
0, &\mbox{ if } |z| \ge \rho.
\end{cases}
\]
The functions $f$ and $g$ satisfy $f(\mu)=g(\mu)=0$, and $f(\sin\theta-\cos\theta) = \cos\theta$ and $g(\sin\theta-\cos\theta) = \sin\theta$ for every $\theta \in [0,\frac{\pi}{2}]$.  Moreover, they are analytic on $\mathbb{C} \setminus \{z : |z|=\rho \}$.  This is an open set containing the spectrum of $B$, since the equality
\[
B = V_1 \begin{pmatrix} S_r-C_r & 0 \\ 0 & \mu I_{n-r} \end{pmatrix} V_1^*
\]
shows that $B$ has spectrum contained in $[-1,1] \cup \{\mu\}$.
The functions $f : \mathbb{C}^{n \times n} \rightarrow \mathbb{C}^{n \times n}$ and $g : \mathbb{C}^{n \times n} \rightarrow \mathbb{C}^{n \times n}$ are thus well-defined in a neighborhood of $B$, with 
\[
f(B) = V_1 \begin{pmatrix} f(S_r-C_r) & 0 \\ 0 & f(\mu I_{n-r}) \end{pmatrix} V_1^* 
= V_1 \begin{pmatrix} C_r & 0 \\ 0 & 0_{n-r} \end{pmatrix} V_1^* 
= H_1
\]
and
\[
g(B) 
= V_1 \begin{pmatrix} g(S_r-C_r) & 0 \\ 0 & g(\mu I_{n-r}) \end{pmatrix} V_1^* 
= V_1 \begin{pmatrix} S_r & 0 \\ 0 & 0_{n-r} \end{pmatrix} V_1^* 
= H_2.
\]
Since $B$ is Hermitian and $\sup_{z \in [-1,1] \cup \{\mu\}} |f'(z)| = 1$,
it follows~\cite[Corollary 3.16]{higham2008functions} that the Fr\'echet derivative $L_f(B,\cdot) : \mathbb{C}^{n \times n} \rightarrow \mathbb{C}^{n \times n}$ of $f$ at $B$ satisfies
\begin{equation} \label{Lf}
\|L_f(B,E)\| \le \|E\|
\end{equation}
for every $E \in \mathbb{C}^{n \times n}$.  Similarly,
\begin{equation} \label{Lg}
\|L_g(B,E)\| \le \|E\|
\end{equation}
for every $E \in \mathbb{C}^{n \times n}$.

For the case in which $r=n$, we instead simply define $f(z) = \frac{1}{2}(z+\sqrt{2-z^2})$ and $g(z) = \frac{1}{2}(-z+\sqrt{2-z^2})$.  Arguments analogous to those above show that in this setting (regardless of the value of $\mu$), $f(B)=H_1$, $g(B)=H_2$, and $L_f$ and $L_g$ satisfy~(\ref{Lf}-\ref{Lg}).

We will now show that the backward error $\widehat{W}_1 \widehat{V}_1 \widehat{C} \widehat{V}_1^* - A_1$ is small.  To begin, denote $D = \mathrm{diag}(\mathrm{diag}(\widehat{V}_1^* \widehat{H}_1 \widehat{V}_1))$ and observe that
\begin{equation*}\begin{split}
\widehat{W}_1 \widehat{V}_1 \widehat{C} \widehat{V}_1^*
&= \widehat{W}_1 \widehat{V}_1 \left( D + \Delta C\right) \widehat{V}_1^* \\
&= \widehat{W}_1 \widehat{V}_1 \widehat{V}_1^* \widehat{H}_1 \widehat{V}_1 \widehat{V}_1^* + \widehat{W}_1 \widehat{V}_1 \left( D - \widehat{V}_1^* \widehat{H}_1 \widehat{V}_1 + \Delta C \right) \widehat{V}_1^* \\
&= \widehat{W}_1 \widehat{H}_1 + \widehat{W}_1 (\widehat{V}_1 \widehat{V}_1^* - I) \widehat{H}_1 \widehat{V}_1 \widehat{V}_1^* + \widehat{W}_1 \widehat{H}_1 (\widehat{V}_1 \widehat{V}_1^* - I) \\&\quad +  \widehat{W}_1 \widehat{V}_1 \left( D - \widehat{V}_1^* \widehat{H}_1 \widehat{V}_1 + \Delta C \right) \widehat{V}_1^*. 
\end{split}\end{equation*}
Since $\widehat{W}_1 \widehat{H}_1 = A_1 + \Delta A_1$, it follows that to leading order,
\begin{equation} \label{backwarderror1step1}
\|\widehat{W}_1 \widehat{V}_1 \widehat{C} \widehat{V}_1^* - A_1\| \le \|\Delta A_1\| + 2 \delta \sigma_1(\widehat{H}_1) + \|D - \widehat{V}_1^* \widehat{H}_1 \widehat{V}_1\| + \|\Delta C\|.
\end{equation}
The next lemma estimates $\|D - \widehat{V}_1^* \widehat{H}_1 \widehat{V}_1\|$.

\begin{lemma} \label{lemma:diagdiff}
To leading order,
\[
\|D - \widehat{V}_1^* \widehat{H}_1 \widehat{V}_1\| \le 2 \left( \|B-\widehat{B}\| + \|\Delta B_2\| + (1+\max\{\mu,1\})\delta + \|\Delta H_1\| \right).
\]
\end{lemma}
\begin{proof}
Since $f(B)=H_1$, we have
\begin{align}
\widehat{V}_1^* \widehat{H}_1 \widehat{V}_1
&= \widehat{V}_1^* f(B) \widehat{V}_1 + \widehat{V}_1^* (\widehat{H}_1-H_1) \widehat{V}_1 \nonumber \\
&= \widehat{V}_1^* \left(f(B)-f(B+E)\right) \widehat{V}_1 + \widehat{V}_1^* f(B+E) \widehat{V}_1 + \widehat{V}_1^* \Delta H_1 \widehat{V}_1, \label{V1H1V1}
\end{align}
for any $E \in \mathbb{C}^{n \times n}$.  Choosing $E = (\widehat{B}+\Delta B_2)\widehat{V}_1^{-*} \widehat{V}_1^{-1} - B$, so that $\widehat{V}_1^{-1}(B+E)\widehat{V}_1 = \widehat{V}_1^{-1} (\widehat{B}+\Delta B_2) \widehat{V}_1^{-*} = \widehat{\Lambda}$, we find that the second term in~(\ref{V1H1V1}) is equal to
\begin{align*}
\widehat{V}_1^* f(B+E) \widehat{V}_1 
&= \widehat{V}_1^* \widehat{V}_1 \widehat{V}_1^{-1} f(B+E) \widehat{V}_1 \\
&= \widehat{V}_1^* \widehat{V}_1 f(\widehat{V}_1^{-1}(B+E) \widehat{V}_1) \\
&= \widehat{V}_1^* \widehat{V}_1 f(\widehat{\Lambda}).
\end{align*}
It follows that
\begin{align*}
f(\widehat{\Lambda}) - \widehat{V}_1^* \widehat{H}_1 \widehat{V}_1
&= \widehat{V}_1^* \left(f(B+E)-f(B)\right) \widehat{V}_1 + (I-\widehat{V}_1^* \widehat{V}_1) f(\widehat{\Lambda}) - \widehat{V}_1^* \Delta H_1 \widehat{V}_1.
\end{align*}
By~(\ref{Lf}), the first term above is bounded by
\begin{align*}
\|\widehat{V}_1^* \left(f(B+E)-f(B)\right) \widehat{V}_1\|
&\le \|E\| \\ 
&= \| (\widehat{B}-B+\Delta B_2)\widehat{V}_1^{-*} \widehat{V}_1^{-1} + B (\widehat{V}_1^{-*} \widehat{V}_1^{-1} - I) \| \\
&\le \|\widehat{B}-B\| + \|\Delta B_2\| +\delta \sigma_1(B)
\end{align*}
to leading order in $\delta$ and $\|E\|$.  
Thus, 
\begin{align*}
\| f(\widehat{\Lambda}) - \widehat{V}_1^* \widehat{H}_1 \widehat{V}_1 \|
&\le \|\widehat{B}-B\| + \|\Delta B_2\| + \delta \sigma_1(B) + \delta \sigma_1(f(\widehat{\Lambda})) + \|\Delta H_1\| \\
&\le \|\widehat{B}-B\| + \|\Delta B_2\| + (1+\max\{\mu,1\})\delta + \|\Delta H_1\|
\end{align*}
where we have used the fact that $\sigma_1(B) \le \max\{\mu,1\}$ and $\sup_{z \in \mathbb{R}\setminus\{\rho\}} |f(z)|=1$, so $\sigma_1(f(\widehat{\Lambda})) \le 1$.
The conclusion follows from the inequality above and the bound
\begin{align*}
\|D - \widehat{V}_1^* \widehat{H}_1 \widehat{V}_1\|
&\le \|D-f(\widehat{\Lambda})\| + \|\widehat{V}_1^* \widehat{H}_1 \widehat{V}_1  - f(\widehat{\Lambda})\| \\
&= \|\mathrm{diag}(\mathrm{diag}(\widehat{V}_1^* \widehat{H}_1 \widehat{V}_1-f(\widehat{\Lambda})))\| + \|\widehat{V}_1^* \widehat{H}_1 \widehat{V}_1 - f(\widehat{\Lambda})\| \\
&\le 2\|\widehat{V}_1^* \widehat{H}_1 \widehat{V}_1 - f(\widehat{\Lambda})\|.
\end{align*}
\end{proof}

The proof of~(\ref{backwarderror1}) is completed by combining Lemma~\ref{lemma:diagdiff} with~(\ref{normDeltaH}),~(\ref{Bdiff}) and~(\ref{backwarderror1step1}), invoking the asymptotic estimates $\sigma_1(A) \sim 1$, $\sigma_1(\widehat{H}_1) \sim 1$, and invoking the inequalities $\|A_1-\widetilde{A}_1\| \le d(A)$, $\|A_2-\widetilde{A}_2\| \le d(A)$.  The proof of~(\ref{backwarderror2}) is almost identical.

\paragraph{Remarks}
Let us recall the discussion in Section~\ref{sec:algorithm}, and reconsider an algorithm based on the eigendecomposition of other choices of $B$, such as $H_1$ or $H_1+H_2$. An attempt to follow the same argument as above to establish stability breaks down, because the resulting functions $f,g$ have unbounded derivatives. 

For later use, we mention now a subtle generalization of Theorem~\ref{thm:backwardstability}.  A careful reading of the proof above shows that~(\ref{backwarderror1}-\ref{backwarderror2}) continue to hold if, for each $i$, the condition $\|\widehat{W}_i^*\widehat{W}_i-I\| \le \delta$ is relaxed to the following pair of conditions:
\begin{align}
\sigma_1(\widehat{W}_i) &= 1 + o(1), \label{relaxedcond1} \\
\|(X_i - \widehat{W}_i) \widehat{H}_i\| &\le \tfrac{1}{2}\delta + o(\eta), \label{relaxedcond2}
\end{align}
where $X_i = \argmin_{X^*X=I} \|\widehat{W}_i-X_i\|$.  In other words, near orthonormality of $\widehat{W}_i$ is not strictly necessary if~(\ref{relaxedcond1}-\ref{relaxedcond2}) can be verified.  This observation will be used in Section~\ref{sec:partial_isometry}.

\section{Algorithms for the polar decomposition and symmetric eigendecomposition} \label{sec:polareig}

A number of iterative algorithms are available for computing the polar decomposition~\cite[\S 8]{higham2008functions}. Among them, the scaled Newton, QDWH and Zolo-pd\footnote{Zolo-pd is observed to be backward stable in experiments; proving it is an open problem.} algorithms are known to be backward stable~\cite{nakatsukasa2012backward,nakatsukasa2016computing}. 

In virtually all algorithms for computing the polar decomposition $A=WH$, the iterates $X_k$ have the same singular vectors as the original $A$, that is, $X_k=Ur_k(\Sigma)V^*$ where $A=U\Sigma V^*$ is the SVD, where $r_k$ is a (odd and usually rational) function. The goal is to find $r_k$ that maps all $\sigma_i(A)$ to 1, as then we have $X_k=UV^*=W$, as required (then $H=W^*A$). Different methods employ different functions $r_k$ to achieve this. 
Here we focus on the Zolo-pd algorithm, as it offers high parallelizability, building upon basic matrix operations (multiplication, QR and Cholesky). 

In Zolo-pd, $r_k(\Sigma)$ is taken to be the type $((2p+1)^k,(2p+1)^k-1)$ best rational approximation to the sign function on $[-\sigma_{\max}(A),-\sigma_{\min}(A)]\cup [\sigma_{\min}(A),\sigma_{\max}(A)]$ (the extremal singular values are estimated unless specified by the user). Such functions are called Zolotarev's functions, and they have the significant property that $r_k$ can be obtained by appropriately composing two Zolotarev functions of types $((2p+1)^{k-1},(2p+1)^{k-1}-1)$ and $(2p+1,2p)$. Combining this fact with a partial fraction representation of rational functions, 
 Zolo-pd  requires just two iterations for convergence in double precision, each iteration involving $p(\leq 8)$ QR or Cholesky factorizations, which can be executed in parallel. For details, see~\cite{nakatsukasa2016computing}. QDWH is a special case where $p=1$, which minimizes the flop count but has the largest number of iterations (and less parallelizability). 

For the symmetric eigenvalue decomposition, 
a spectral divide-and-conquer algorithm~\cite{nakatsukasa2013stable}
 can be developed also based on the polar decomposition. 
The idea is that for a symmetric matrix $B$, the unitary polar factor $W$ 
is equivalent to the matrix sign decomposition, and 
has eigenvalues $\pm 1$ since 
the SVD and eigendecomposition for a symmetric $B$ are related by 
$B = U\Sigma V^* = (US) (S\Sigma) V^* =V\Lambda V^* $
for $S=\mbox{diag}(\pm 1)$, so $U=VS$, and hence  
$W=UV^*=VSV^* =: [V_+\ V_-]\big[
\begin{smallmatrix}
I_{n_+} & \\  & -I_{n-n_+}
\end{smallmatrix}
\big][V_+\ V_-]^*$, where $n_+$ is the number of positive eigenvalues in $B$
(we work with a shifted matrix $B-sI$ so that $n_+\approx n/2$). 
Thus $\frac{1}{2}(W+I) = V_+V_+^*$ 
is a partial isometry onto the eigenspace corresponding to the positive 
eigenvalues of $B$. 
From this we can obtain 
 an orthogonal transformation $\widetilde V:=[V_+\ V_-]$ that block diagonalizes $B$. 
We perform this process recursively on the decoupled diagonal blocks to diagonalze the matrix, resulting in the full eigendecomposition. 
It transpires that the overal process is backward stable if each polar decomposition is computed in a backward stable way, analogous to Theorem~\ref{thm:backwardstability}. 

Of course, classical algorithms for symmetric eigendecomposition based on reduction to tridiagonal form~\cite[Ch.~8]{golub2012matrix} are both effective and stable. 
An advantage of algorithms based on the polar 
decomposition is that they can be implemented in a communication-minimizing manner using only BLAS-3 operations such as matrix multiplication,
QR factorization, and Cholesky decomposition. 


\subsection{Dealing with ill-conditioned matrices} \label{sec:illcond}
The matrices $A_i$ can be ill-conditioned, and this impacts the polar decompositions in lines \ref{line:polar1}-\ref{line:polar2} of Algorithm~\ref{alg:csd}: 
 $W_1 H_1 = A_1$ is ill-conditioned when there exists $\theta_i\approx \tfrac{\pi}{2}$, and  $W_2 H_2 = A_2$ when $\theta_i\approx 0$. When the conditioning is $O(u^{-1})$ or larger, Zolo-pd becomes expensive and even unreliable, as Zolotarev's function of type at most $(17^2,17^2-1)$ may not be enough to map the $O(u)$ singular values to 1. Below we assume the use of double precision arithmetic; the value of $p$ (here $p=8$) will depend on the unit roundoff $u$ (though rather weakly, like $\sqrt{|\log(u)|}$). 


Here we discuss a remedy for such situations.
When computing the polar decompositions, we apply the Zolo-pd algorithm, but working with the narrower interval 
$[-1,-\epsilon]\cup [\epsilon,1]$, 
rather than $[-1,-\sigma_{\min}(A_1)]\cup [\sigma_{\min}(A_1),1]$. We choose $\epsilon$ to be a modest multiple of unit roundoff; here $\epsilon=10^{-15}$. 
The resulting modified Zolo-pd computes 
$\widetilde W_i:=U_ir_2(\Sigma_i)V_1^*$, where 
$r_2$ is the Zolotarev function of type $(17^2,17^2-1)$ on 
$[-1,-\epsilon]\cup [\epsilon,1]$. 
In particular, this gives 
$r_2(x)=1-O(u)$ for $x\in [\epsilon,1]$, and $0\leq r_2(x)\leq 1$ on $[0,1]$. 
It follows that $\widetilde H_i:=\widetilde W_i^*A_i=V_1\Sigma_i r_2(\Sigma_i)V_1^*$ has eigenvalues 
$\lambda_j(\widetilde H_i)\in\sigma_j(A_i)+[-\tilde\epsilon,0]$, where 
$\tilde\epsilon=O(u)$. 
We also have 
\begin{equation}
  \label{eq:hh}
\|\widetilde H_i-H_i\|=\|\Sigma_i r_2(\Sigma_i)-\Sigma_i\|  , 
\end{equation}
 where 
$\Sigma_i r_2(\Sigma_i)-\Sigma_i$ is diagonal with $(j,j)$ element
$a_j:=\sigma_j(A_i) r_2(\sigma_j(A_i))-\sigma_j(A_i)$. 
We claim that $|a_j|=O(u)$ for all $j$: indeed, for $j$ such that 
$\sigma_j(A_i)>\epsilon$, we have $r_2(\sigma_j(A_i))=1-O(u)$, so 
$|a_j|=O(u)$. If $\sigma_j(A_i)\leq \epsilon$, then 
$|a_j|\leq |\sigma_j(A_i)|\leq \epsilon=O(u)$. 
Together with~\eqref{eq:hh} we obtain $\|\widetilde H_i-H_i\|=O(u)$. 
Moreover, we have $\|\widetilde W_i\widetilde H_i-A\|=\|r_2(\Sigma_i)^2\Sigma_i-\Sigma_i\|$, which is also $O(u)$ by a similar argument. Summarizing, we have
\begin{equation}  \label{eq:HHWH}
\|\widetilde H_i-H_i\|=O(u),\qquad  \|\widetilde W_i\widetilde H_i-A\|=O(u)  , 
\end{equation}
 even though $\widetilde W_i\widetilde H_i$ is not a polar decomposition since $\widetilde W_i$ is not orthogonal. 

In view of the first equation in~\eqref{eq:HHWH}, we proceed to lines 3--5 in Algorithm~\ref{alg:csd}, which gives results that are backward stable. The only issue lies in lines 6--7, where we would compute $\widetilde W_i V_1$: since $\widetilde W_i$ does not have orthonormal columns, neither does $\widetilde W_i V_1$. 

To overcome this issue, we make the following observation: the $R$-factors in the QR factorizations of $A_i$ and $H_i$ are identical, 
that is, $A_i=Q_iR_i$ and $H_i=Q_{i,H}R_i$
 (we adopt the convention that the diagonal elements of $R$ are nonnegative; this makes the QR factorization unique~\cite[Section 5.2.1]{golub2012matrix}). It follows that 
in the QR and polar decompositions $A_i=Q_iR_i=W_iH_i$, we have the relation 
$W_i=Q_iQ_{i,H}^*$. 

Recalling from~\eqref{eq:HHWH} that $\widetilde H_i=H_i+O(u)$, this suggests the following: before line 6 of Algorithm~\ref{alg:csd}, compute the QR factorizations 
$A_i=Q_iR_i$ and $\widetilde H_i=\widetilde Q_{i,H}\widetilde R_i$, and 
redefine $W_i:=Q_i\widetilde{Q}_{i,H}^*$. 
We summarize the process in Algorithm~\ref{alg:csdillcond}. 

\begin{breakablealgorithm} \label{alg:csdillcond}
\caption[CS decomposition, ill-conditioned case]{Modification to Algorithm~\ref{alg:csd} when $A_1$ and/or $A_2$ is ill-conditioned}
\begin{algorithmic}[1]
\STATE{In place of lines 1--2 of Algorithm~\ref{alg:csd}, 
compute $A_i\approx \widetilde W_i \widetilde H_i$, obtained by a modified Zolo-pd mapping the interval $[-1,-\epsilon]\cup [\epsilon,1]$ to 1 ($\epsilon=10^{-15}$ in double precision)}
\STATE{Compute QR factorizations $A_i=Q_iR_i$ and 
$\widetilde H_i=\widetilde Q_{i,H}\widetilde R_i$} 
\STATE{Set $W_i=Q_i\widetilde{Q}_{i,H}^*$, $H_i=\widetilde H_i$}
\STATE{Proceed with lines 3 onwards of Algorithm~\ref{alg:csd}}
\end{algorithmic}
\end{breakablealgorithm}\vspace{0.05in}

In practice, since it is usually unknown a priori if $A_i$ is ill-conditioned, we execute Zolo-pd as usual, in which one of the preprocessing step is to estimate $\sigma_{\min}(A_i)$: if it is larger than $10^{-15}$, we continue with standard Zolo-pd; otherwise $A_i$ is ill conditioned, and we run Algorithm~\ref{alg:csdillcond} (this is necessary only for $i$ for which $A_i$ is ill conditioned). 

An important question is whether this process is stable. 
Since 
$\|\widetilde H_i-H_i\|=O(u)$ by \eqref{eq:HHWH} and 
the resulting $W_i$ is orthogonal to working precision by construction, 
the main question is whether $W_i\widetilde H_i$ still gives a backward stable polar decomposition for $A_i$, that is, whether $\| W_i\widetilde H_i-A\|=O(u)$ holds or not.
To examine this, note that 
 \begin{align*}
\| W_i\widetilde H_i-A\|&=
\| Q_i\widetilde Q_{i,H}^*\widetilde H_i-Q_iR_i\|=
\| Q_i\widetilde R_i-Q_iR_i\|=\| \widetilde R_i-R_i\|. 
 \end{align*}
 The question therefore becomes whether $\| \widetilde R_i-R_i\|=O(u)$
 holds (note that this is an inexpensive condition to check).  In general, the triangular factor in the QR factorization can
 be ill conditioned; however, it is known to be usually much better
 conditioned than the original
 matrix~\cite{chang1997perturbation}. Indeed in all our experiments
 with ill-conditioned $A_i$, we observed that
 $\| \widetilde R_i-R_i\|$ was a small multiple of unit roundoff, 
 indicating Algorithm~\ref{alg:csdillcond} is an effective workaround for ill-conditioned problems. 

In the rare event that $\|\widetilde R_i-R_i\|$ is unacceptably large, 
a conceptually simple and robust (but expensive) workaround is to compute the polar decompositions via the SVD, which is unconditionally backward stable.

\subsection{Dealing with rank deficiency} \label{sec:partial_isometry}
When $A$ is a rank deficient partial isometry ($r<n$) rather than having orthonormal columns, a natural goal is to compute the economical decomposition \eqref{eq:partCSecon}, as it saves memory requirement and allows efficient operations such as multiplications.

Recall that the rank $r$ of $A$ can be computed via $\|A\|_F\approx \sqrt{r}$; here we assume that $r<n$. When $r<n$, both $A_1$ and $A_2$ are singular. 
As described in~\cite{nakatsukasa2013stable}, 
QDWH and Zolo-pd are capable of computing the canonical polar decomposition, and in exact arithmetic it computes $A_i=W_iH_i$ where $W_i$ are partial isometries. In finite-precision arithmetic, however, roundoff error usually causes the zero singular values of $A_i$ to get mapped to nonzero values. These eventually converge to 1 in QDWH (in six iterations), but Zolo-pd, which terminates in two iterations, faces the same difficulty discussed above, 
usually producing a $W_i$ that has $n-r$ singular values that are between 0 and 1. Then the computed $\widehat W_i$ does not have singular values that are all close to 0 or 1, and hence neither does the resulting $U_i$. Here we discuss a modification of Zolo-pd to deal with such issues. 

The first step identical to the previous subsection: 
we invoke the modified Zolo-pd to map singular values in $[\epsilon,1]$ to 1. Recall that the resulting $H_i$ are correct to working precision. 
We then compute $B = H_2-H_1 + \mu(I-A^*A)$ with $\mu=2$, 
and its eigendecomposition $B=V_1 \Lambda V_1^*$ as usual, and proceed as described in the second-to-last remark after Algorithm~\ref{alg:csd} to extract the relevant matrices. 
We summarize the process in Algorithm~\ref{alg:csdrankdef}. 

\begin{breakablealgorithm} \label{alg:csdrankdef}
\caption[CS decomposition, rank-deficient case]{Modification to Algorithm~\ref{alg:csd} when $A$ is rank deficient}
\begin{algorithmic}[1]
\STATE{Compute $A_i\approx \widetilde W_i \widetilde H_i$ for $i=1,2$, obtained by a modified Zolo-pd mapping the interval $[-1,-\epsilon]\cup [\epsilon,1]$ to 1
}\label{line:pdrankdef}
\STATE{$W_i =\widetilde W_i,  H_i= \widetilde H_i$ } \label{line:Wirankdef}
\STATE{$B = H_2-H_1 + 2(I-A^*A)$} \label{line:Brankdef}
\STATE{$V_1 \Lambda V_1^* = B$ (symmetric eigendecomposition)} \label{line:eigrankdef}
\STATE{Find $V_{1r}\in\mathbb{C}^{n \times r}$, the eigenvectors corresponding to eigenvalues in $[-1,1]$} \label{line:eigpartrankdef}
\STATE{$U_1 = W_1 V_{1r}$} \label{line:U1rankdef}
\STATE{$U_2 = W_2 V_{1r}$} \label{line:U2rankdef}
\STATE{$C = \mathrm{diag}(\mathrm{diag}(V_{1r}^* H_1 V_{1r}))$} \label{line:Crankdef}
\STATE{$S = \mathrm{diag}(\mathrm{diag}(V_{1r}^* H_2 V_{1r}))$} \label{line:Srankdef}
\RETURN $U_1,U_2,C,S,V_1:=V_{1r}$
\end{algorithmic}
\end{breakablealgorithm}\vspace{0.05in}

Note that the modified Zolo-pd is used in exactly the same way 
in Algorithms~\ref{alg:csdillcond} and \ref{alg:csdrankdef}. 
This lets us easily treat the situation where both issues are present: 
$A$ is rank-deficient and $C$ or $S$ is ill conditioned. 
In this case, we 
replace line 2 of Algorithm~\ref{alg:csdrankdef} by lines 2--3 of Algorithm \ref{alg:csdillcond}.
We also note that in spectral divide-and-conquer algorithms such as Zolo-eig, it is straightforward to modify the algorithm to compute only eigenpairs lying in the prescribed interval $[-1,1]$ (by splitting e.g. at $1.1$), thus saving some cost. 

A nontrivial question here is: does each $U_i$ have orthonormal columns? This is not obvious because $W_i$ is not orthonormal. 
To examine this, we momentarily abuse notation by writing $U_i$ as $U_{i,r}$ to distinguish it from the $U_i$ from earlier sections ($U_{i,r}$ consists of the first $r$ columns of $U_i$).  
Now recall that $W_i=U_ir_2(\Sigma_i)V_{1}^*$, which we rewrite as 
\[
W_i=[U_{i,r},U_{i,r}^\perp]
\begin{pmatrix}
r_2(\Sigma_{i,r})  & \\
  & r_2(\Sigma_{i,>r})
\end{pmatrix}[V_{1r},V_{1r}^\perp]^*
\]
where $\Sigma_{i,r}$ are the $r$ (nonzero) leading singular values of $A_i$. 
Assuming that they are larger than $\epsilon$ (when this is violated we invoke Algorithm~\ref{alg:csdillcond} as mentioned above), 
we have 
$r_2(\Sigma_{i,r})=I_r+O(u)$. Therefore 
$W_iV_{1r}=U_{i,r}+O(u)$, which is orthonormal to working precision, as required. 
Our experiments illustrate that $U_i$ are indeed orthonormal to working precision. 


A second question that must be addressed is whether Algorithm~\ref{alg:csdrankdef} is still backward stable, given that $\|W_i^* W_i - I\|$ may be much greater than $u$.  To answer this question, we recall the remark made at the end of Section~\ref{sec:backwardstability}: backward stability is still ensured if $\sigma_1(W_i) \le 1+O(u)$ and $\|(X_i-W_i)H_i\| = O(u)$, where $X_i = \argmin_{X^*X=I} \|X-W_i\|$. The first of these conditions is clearly satisfied.  For the second, observe that $H_i = W_i^* A = V_1 r_2(\Sigma_i) \Sigma_i V_1^*$ and by~\cite[Theorem 8.4]{higham2008functions}, $X_i=U_i V_1^*$.  Thus,
\begin{align*}
(X_i-W_i)H_i 
&= U_i (I-r_2(\Sigma_i)) r_2(\Sigma_i) \Sigma_i V_1^* \\
&= U_i
\begin{pmatrix}
(I_r - r_2(\Sigma_{i,r})) r_2(\Sigma_{i,r}) \Sigma_{i,r}  & \\
  & (I_{n-r} - r_2(\Sigma_{i,>r})) r_2(\Sigma_{i,>r}) \Sigma_{i,>r}
\end{pmatrix}V_1^*.
\end{align*}
In the block diagonal matrix above, the upper left block is $O(u)$ because $\|I_r - r_2(\Sigma_{i,r})\| = O(u)$ and $\|\Sigma_{i,r}r_2(\Sigma_{i,r})\|\leq 1+O(u)$, whereas the lower right block is $O(u)$ because $\|\Sigma_{i,>r}\| = O(u)$ and $\|(I_{n-r} - r_2(\Sigma_{i,>r})) r_2(\Sigma_{i,>r})\|\leq 1+O(u)$.  We conclude that $\|(X_i-W_i)H_i\| = O(u)$, as desired.

We emphasize that all of the arguments above hinge upon the assumption that $r_2(A_i)$ is computed in a backward stable manner in the Zolo-pd algorithm.  This is supported by extensive numerical evidence but not yet by a proof~\cite{nakatsukasa2016computing}.

\section{Numerical examples} \label{sec:numerical}

\begin{table}[t]
\centering
\begin{filecontents*}{Data/fullrank.dat}
   3.0000000e+01   4.4408921e-16   3.3433056e+00   2.0627330e+01   1.4912651e+01   2.6447960e+01   1.4732856e+01   4.0058256e+01   4.8891321e+00   2.3232124e+01
   4.2000000e+01   6.6613381e-16   2.2347332e+00   1.5860877e+01   1.7842621e+01   3.2318679e+01   1.4348726e+01   3.6484383e+01   5.5069547e+00   4.2120050e+01
   6.0000000e+01   4.4408921e-16   3.4238612e+00   2.4433409e+01   1.2099421e+01   4.3964961e+01   1.3257997e+01   3.7120175e+01   5.3874638e+00   4.8553748e+01
   8.5000000e+01   5.5511151e-16   3.6839081e+00   1.4646871e+01   1.9820996e+01   5.1025883e+01   2.0539506e+01   4.3643931e+01   5.9653733e+00   5.2558247e+01
   1.2000000e+02   4.4408921e-16   4.7903645e+00   2.6799497e+01   1.7643911e+01   5.2085213e+01   2.2855736e+01   6.2293547e+01   7.6195502e+00   5.7636997e+01
   1.7000000e+02   1.7763568e-15   1.4087410e+00   7.7035333e+00   2.9312095e+01   6.4285440e+01   2.8626842e+01   7.2246177e+01   8.5482624e+00   7.2111801e+01
   2.4000000e+02   2.2204460e-15   1.1987306e+00   6.9208128e+00   2.4012437e+01   8.3925305e+01   2.3936533e+01   7.2896046e+01   1.0200331e+01   1.0209441e+02
   3.3900000e+02   3.7747583e-15   7.4912541e-01   4.2610179e+00   3.0535472e+01   8.3149520e+01   3.3807775e+01   9.9917150e+01   9.7837388e+00   9.2587734e+01
   4.8000000e+02   4.2188475e-15   7.0074900e-01   4.6808475e+00   2.4529152e+01   9.7897852e+01   2.6101878e+01   1.2897454e+02   1.1277388e+01   1.1340636e+02
   6.7900000e+02   3.3306691e-15   9.4923470e-01   1.2149644e+01   2.2776990e+01   1.5724307e+02   2.9363059e+01   1.1841197e+02   1.1452894e+01   1.4803682e+02
   3.0000000e+01   9.9995945e-10   1.1152321e+00   1.6490761e+00   1.3676141e+01   2.7986835e+01   1.2591448e+01   2.9929104e+01   4.6359458e+00   2.6632280e+01
   4.2000000e+01   1.1995818e-09   1.1215180e+00   1.6388704e+00   1.6472937e+01   3.3997745e+01   1.2918507e+01   2.9399151e+01   5.1346445e+00   3.2326213e+01
   6.0000000e+01   1.4923931e-09   1.1190146e+00   1.6111671e+00   1.7342986e+01   3.6486322e+01   1.9178441e+01   4.4510220e+01   5.7045936e+00   4.2798138e+01
   8.5000000e+01   1.8042960e-09   1.1198629e+00   1.6603508e+00   1.5080499e+01   4.8388363e+01   1.3934736e+01   4.6391208e+01   6.1421218e+00   4.8702344e+01
   1.2000000e+02   2.1018940e-09   1.1343841e+00   1.6567894e+00   1.8389600e+01   5.6576634e+01   2.0434680e+01   6.7830397e+01   7.4183800e+00   5.7907518e+01
   1.7000000e+02   2.5205222e-09   1.1306378e+00   1.6352872e+00   1.8998688e+01   6.2151336e+01   1.8102563e+01   7.3144778e+01   8.7379578e+00   7.6034005e+01
   2.4000000e+02   3.1114329e-09   1.1171113e+00   1.6141551e+00   2.5992923e+01   7.2076647e+01   2.3779987e+01   8.9138997e+01   9.8995878e+00   7.7958108e+01
   3.3900000e+02   3.6439658e-09   1.1250066e+00   1.6557287e+00   1.9582745e+01   8.8790846e+01   2.2955202e+01   9.7806190e+01   9.8484146e+00   1.0261531e+02
   4.8000000e+02   4.3265387e-09   1.1294129e+00   1.6569154e+00   2.5081988e+01   9.6850570e+01   2.9182771e+01   1.1402675e+02   1.1623040e+01   1.1554470e+02
   6.7900000e+02   5.2089910e-09   1.1300992e+00   1.6301306e+00   2.4996983e+01   1.1716444e+02   2.2585954e+01   1.2954926e+02   1.1476175e+01   1.3096905e+02
   3.0000000e+01   5.5511151e-16   4.0137011e+00   1.5108113e+01   1.9308990e+01   2.8521541e+01   2.2946777e+01   2.1807888e+01   4.6759923e+00   2.1440991e+01
   4.2000000e+01   5.5511151e-16   5.6744753e+00   1.6854223e+01   9.7078825e+00   2.8921020e+01   1.6049660e+01   2.3762702e+01   4.8363213e+00   2.7745739e+01
   6.0000000e+01   4.4408921e-16   9.5233983e+00   1.7151736e+01   2.2900489e+01   3.7845815e+01   6.4054494e+00   4.2611429e+01   5.5330480e+00   3.7054266e+01
   8.5000000e+01   4.4408921e-16   1.0509153e+01   5.4967358e+01   1.7915389e+01   4.4894362e+01   7.0550181e+00   4.2790569e+01   6.0175513e+00   3.9913817e+01
   1.2000000e+02   4.4408921e-16   1.1798843e+01   6.0012013e+01   1.8244930e+01   5.1919359e+01   8.8630085e+00   4.8473968e+01   7.7780085e+00   4.6755872e+01
   1.7000000e+02   1.3322676e-15   4.0108604e+00   1.8559663e+01   3.1754717e+01   7.2345169e+01   2.2199311e+01   6.3854843e+01   9.1501538e+00   7.2095826e+01
   2.4000000e+02   2.2204460e-15   3.6402457e+00   1.0388278e+01   2.4842091e+01   7.4187560e+01   1.1482153e+01   7.4364571e+01   9.9475717e+00   6.8583271e+01
   3.3900000e+02   2.6645353e-15   2.8104707e+00   1.4907147e+01   2.7595169e+01   1.1443962e+02   2.0811490e+01   1.9552273e+02   9.7377656e+00   1.0160699e+02
   4.8000000e+02   3.2196468e-15   3.0920595e+00   1.0157200e+01   1.2466074e+01   1.0885465e+02   1.2383008e+01   1.1805389e+02   1.1298251e+01   1.3792641e+02
   6.7900000e+02   3.3306691e-15   3.0467406e+00   9.4068979e+00   3.3613694e+01   1.2626124e+02   1.2632466e+01   1.4925870e+02   1.1515778e+01   1.3983894e+02
   3.0000000e+01   9.8019104e-10   1.2838250e+00   1.6948173e+00   2.2993401e+01   2.7132546e+01   1.1249871e+01   2.0797051e+01   5.2328048e+00   2.1231525e+01
   4.2000000e+01   1.2098769e-09   1.3020064e+00   1.6816816e+00   1.8045879e+01   3.0548777e+01   1.1990738e+01   3.0918558e+01   5.0030318e+00   3.0141376e+01
   6.0000000e+01   1.4620862e-09   1.2494845e+00   1.6907748e+00   1.6630010e+01   4.0001706e+01   1.2796974e+01   3.7029165e+01   5.4711111e+00   3.8467895e+01
   8.5000000e+01   1.8273302e-09   1.1367037e+00   1.6077616e+00   1.4870233e+01   5.3310338e+01   1.3565410e+01   4.9999088e+01   6.5209499e+00   4.8854931e+01
   1.2000000e+02   2.1187772e-09   1.2595692e+00   1.6606674e+00   2.0438303e+01   4.9257694e+01   1.6405878e+01   6.7558326e+01   7.5827581e+00   4.8719705e+01
   1.7000000e+02   2.5716143e-09   1.1645690e+00   1.6577772e+00   2.5629507e+01   6.3212887e+01   1.9444369e+01   6.5301828e+01   8.6239324e+00   6.5342046e+01
   2.4000000e+02   3.0054768e-09   1.1755619e+00   1.6590041e+00   2.3270371e+01   8.2380538e+01   2.2937899e+01   8.5599754e+01   9.9408342e+00   8.6370982e+01
   3.3900000e+02   3.6183689e-09   1.1588843e+00   1.6497764e+00   2.6448410e+01   1.0993200e+02   2.1097661e+01   1.0169844e+02   9.8946934e+00   8.5354560e+01
   4.8000000e+02   4.3193258e-09   1.1714016e+00   1.6613475e+00   2.5755438e+01   1.1643885e+02   2.7436033e+01   1.2070830e+02   1.1268147e+01   1.0499166e+02
   6.7900000e+02   5.1751452e-09   1.1439864e+00   1.6415003e+00   2.3757707e+01   1.4006775e+02   2.9236642e+01   1.3062624e+02   1.1670433e+01   1.5659699e+02
\end{filecontents*}
\pgfplotstabletypeset[
every head row/.style={
before row={\toprule
 &  &  &  &\multicolumn{4}{c|}{$\frac{\|\widehat{A}-A\|_2}{d(A)}$} &\multicolumn{4}{c|}{$\frac{\|\widehat{U}_1^*\widehat{U}_1-I\|_2}{u}$} &\multicolumn{4}{c|}{$\frac{\|\widehat{U}_2^*\widehat{U}_2-I\|_2}{u}$} &\multicolumn{4}{c|}{$\frac{\|\widehat{V}_1^*\widehat{V}_1-I\|_2}{u}$}  \\},
after row=\midrule},
every nth row={10}{before row=\midrule},
every last row/.style={after row=\bottomrule},
columns={leftcol,0,1,2,3,4,5,6,7,8,9},
create on use/leftcol/.style={column name={Test class},create col/set list={\multirow{10}{*}{1},\multirow{29}{*}{1'},\multirow{48}{*}{2},\multirow{67}{*}{2'}}},
columns/leftcol/.style={string type,column type/.add={|}{|},column name={Test}},
columns/0/.style={precision=1,verbatim,column type/.add={}{|},column name={$n$}},
columns/1/.style={dec sep align={c|},sci,sci 10e,sci zerofill,precision=1,column type/.add={}{|},column name={$d(A)$}}, 
columns/2/.style={column type={c},dec sep align,zerofill,column type/.add={}{},column name={Z}}, 
columns/3/.style={column type={c},dec sep align={c|},zerofill,column type/.add={}{|},column name={L}},
columns/4/.style={column type={c},dec sep align,zerofill,column type/.add={}{},column name={Z}}, 
columns/5/.style={column type={c},dec sep align={c|},zerofill,column type/.add={}{|},column name={L}},
columns/6/.style={column type={c},dec sep align,zerofill,column type/.add={}{},column name={Z}}, 
columns/7/.style={column type={c},dec sep align={c|},zerofill,column type/.add={}{|},column name={L}},
columns/8/.style={column type={c},dec sep align,zerofill,column type/.add={}{},column name={Z}}, 
columns/9/.style={column type={c},dec sep align={c|},zerofill,column type/.add={}{|},column name={L}},
]
{Data/fullrank.dat}
\caption{Residuals and orthogonality measures for Zolo-csd (Z) and LAPACK (L) on the test matrices~(\ref{test1}),~(\ref{test1}'),~(\ref{test2}), and~(\ref{test2}').}
\label{tab:fullrank}
\end{table}

\begin{table}[t]
\centering
\begin{filecontents*}{Data/rankdef.dat}
   3.0000000e+01   6.6613381e-16   7.2776865e+00   4.0345900e+00   4.5286557e+00   4.2825827e+00
   4.2000000e+01   4.6312155e-16   1.5989823e+01   5.2590484e+00   5.3351630e+00   4.3724678e+00
   6.0000000e+01   8.8817842e-16   7.9848707e+00   5.2266505e+00   5.1712260e+00   5.3799775e+00
   8.5000000e+01   6.0795120e-16   2.1775725e+01   5.7259727e+00   5.9859715e+00   5.4245969e+00
   1.2000000e+02   7.0580040e-16   5.2885969e+01   6.7561344e+00   6.6984840e+00   6.6615926e+00
   1.7000000e+02   7.9891842e-16   6.2661938e+01   8.2482237e+00   8.3232325e+00   7.8006036e+00
   2.4000000e+02   8.8817842e-16   3.4867921e+01   9.7037946e+00   9.4431065e+00   8.5299810e+00
   3.3900000e+02   1.1102230e-15   3.0223680e+01   9.2118826e+00   9.4341591e+00   8.3129831e+00
   4.8000000e+02   1.8873791e-15   2.7899689e+01   1.0610406e+01   1.0710061e+01   9.6534091e+00
   6.7900000e+02   2.5535130e-15   8.4958019e+01   1.1063617e+01   1.1118041e+01   1.0056223e+01
   3.0000000e+01   1.1696100e-09   2.3146427e+00   2.3532933e+01   2.2589593e+01   4.2185863e+00
   4.2000000e+01   1.3300133e-09   2.5109729e+00   2.4878437e+01   2.5770890e+01   4.4316525e+00
   6.0000000e+01   1.6930611e-09   2.4254945e+00   2.5746605e+01   2.3798374e+01   4.8898291e+00
   8.5000000e+01   2.0336079e-09   2.4704778e+00   2.6566722e+01   2.6129224e+01   5.7982274e+00
   1.2000000e+02   2.4270719e-09   2.4495587e+00   2.6682064e+01   2.7328547e+01   6.9720761e+00
   1.7000000e+02   2.9480317e-09   2.4087089e+00   2.8572869e+01   2.8364467e+01   7.3640955e+00
   2.4000000e+02   3.4184696e-09   2.4685098e+00   3.0633517e+01   3.0436419e+01   8.7382094e+00
   3.3900000e+02   4.1592890e-09   2.4492235e+00   2.9865202e+01   2.9676332e+01   8.4562384e+00
   4.8000000e+02   4.9922391e-09   2.3994180e+00   3.1800546e+01   3.1511408e+01   9.3541261e+00
   6.7900000e+02   5.8696866e-09   2.4730032e+00   3.1563650e+01   3.1709690e+01   1.0180543e+01
   3.0000000e+01   4.4408921e-16   1.0240016e+01   4.2325806e+00   4.6418502e+00   4.7789946e+00
   4.2000000e+01   6.6613381e-16   1.1624066e+01   4.2325416e+00   4.6117007e+00   4.3903782e+00
   6.0000000e+01   6.0631968e-16   2.2434057e+01   5.1339357e+00   5.3042571e+00   5.1232796e+00
   8.5000000e+01   6.6613381e-16   2.2746273e+01   6.1203502e+00   5.6427139e+00   5.7612424e+00
   1.2000000e+02   7.2457995e-16   2.3359228e+01   6.4941946e+00   7.1061753e+00   6.1823036e+00
   1.7000000e+02   8.0610809e-16   1.8641533e+01   8.3088876e+00   8.4354175e+00   7.4540335e+00
   2.4000000e+02   9.6443162e-16   3.1041183e+01   9.5050653e+00   9.2655584e+00   8.9064420e+00
   3.3900000e+02   1.0652114e-15   4.1149900e+01   9.2915560e+00   9.4289156e+00   8.5604239e+00
   4.8000000e+02   1.7763568e-15   2.7755168e+01   1.0627887e+01   1.0391918e+01   9.9196975e+00
   6.7900000e+02   1.5543122e-15   3.3128864e+01   1.0904720e+01   1.0984820e+01   1.0188361e+01
   3.0000000e+01   1.1370296e-09   2.1869612e+00   2.1777032e+01   2.7264867e+01   3.9758359e+00
   4.2000000e+01   1.2858899e-09   3.2103721e+00   2.2120895e+01   1.7884180e+01   4.8937761e+00
   6.0000000e+01   1.6424935e-09   2.6699681e+00   2.7885670e+01   2.0364473e+01   5.3055916e+00
   8.5000000e+01   2.0480889e-09   2.2420666e+00   1.9577157e+01   2.7848082e+01   5.4959711e+00
   1.2000000e+02   2.4390045e-09   2.4101949e+00   2.9123421e+01   2.6101701e+01   6.3420402e+00
   1.7000000e+02   2.9172119e-09   2.5950188e+00   3.1461828e+01   2.6970176e+01   7.8637134e+00
   2.4000000e+02   3.4343141e-09   2.3611700e+00   2.9081822e+01   3.1944468e+01   8.8650007e+00
   3.3900000e+02   4.0865300e-09   2.4285652e+00   2.7055121e+01   2.7088329e+01   8.4954677e+00
   4.8000000e+02   4.9427421e-09   2.4962056e+00   3.1653726e+01   3.1675743e+01   9.7505237e+00
   6.7900000e+02   5.8514138e-09   2.5541639e+00   3.3870932e+01   3.1075051e+01   1.0082390e+01
\end{filecontents*}
\pgfplotstabletypeset[
every head row/.style={
output empty row,
before row={\toprule
Test & $n$ & \multicolumn{2}{c|}{$d(A)$} &\multicolumn{2}{c|}{$\frac{\|\widehat{A}-A\|_2}{d(A)}$} &\multicolumn{2}{c|}{$\frac{\|\widehat{U}_1^*\widehat{U}_1-I\|_2}{u}$} &\multicolumn{2}{c|}{$\frac{\|\widehat{U}_2^*\widehat{U}_2-I\|_2}{u}$} &\multicolumn{2}{c|}{$\frac{\|\widehat{V}_1^*\widehat{V}_1-I\|_2}{u}$}  \\},
after row=\midrule},
every nth row={10}{before row=\midrule},
every last row/.style={after row=\bottomrule},
columns={leftcol,0,1,2,3,4,5},
create on use/leftcol/.style={column name={Test class},create col/set list={\multirow{10}{*}{3},\multirow{29}{*}{3'},\multirow{48}{*}{4},\multirow{67}{*}{4'}}},
columns/leftcol/.style={string type,column type/.add={|}{|}},
columns/0/.style={precision=1,verbatim,column type/.add={}{|}},
columns/1/.style={dec sep align={c|},sci,sci 10e,sci zerofill,precision=1,column type/.add={}{|}}, 
columns/2/.style={column type={c},dec sep align={c|},fixed zerofill,column type/.add={}{|}}, 
columns/3/.style={column type={c},dec sep align={c|},fixed zerofill,column type/.add={}{|}},
columns/4/.style={column type={c},dec sep align={c|},fixed zerofill,column type/.add={}{|}}, 
columns/5/.style={column type={c},dec sep align={c|},fixed zerofill,column type/.add={}{|}}
]
{Data/rankdef.dat}
\caption{Residuals and orthogonality measures for Zolo-csd on the test matrices~(\ref{test3}),~(\ref{test3}'),~(\ref{test4}), and~(\ref{test4}').}
\label{tab:rankdef}
\end{table}

We tested Algorithm~\ref{alg:csd} on the following examples adapted from~\cite{sutton2013divide}.  Below, $\mathrm{nint}(x)$ denotes the nearest integer to a real number $x$, and $St(n,m) = \{A \in \mathbb{C}^{m \times n} \mid A^*A=I\}$ denotes the complex Stiefel manifold.
\begin{enumerate}
\item \label{test1} (Haar) A $2n \times n$ matrix sampled randomly from the Haar measure on $St(n,2n)$.
\item \label{test2} (Clustered) A $2n \times n$ matrix $A = \left( \begin{smallmatrix} U_1 C V_1^* \\ U_2 S V_1^* \end{smallmatrix} \right)$, where $U_1,U_2,V_2 \in \mathbb{C}^{n \times n}$ are sampled randomly from the Haar measure on $St(n,n)$, and $C$ and $S$ are generated with the following MATLAB commands:
\begin{lstlisting}[style=Matlab-editor]
delta = 10^(-18*rand(n+1),1);
theta = pi/2*cumsum(delta(1:n))/sum(delta);
C = diag(cos(theta)); 
S = diag(sin(theta));
\end{lstlisting}
This code tends to produce principal angles $\theta_1,\theta_2,\dots,\theta_n$ that are highly clustered.
\item \label{test3} (Rank-deficient, Haar) A $2n \times n$ matrix of rank $r=\mathrm{nint}(3n/4)$ given by $A=XY^*$, where $X \in \mathbb{C}^{2n \times r}$ and $Y \in \mathbb{C}^{n \times r}$ are sampled randomly from the Haar measure on $St(r,2n)$ and $St(r,n)$, respectively.
\item \label{test4} (Rank-deficient, clustered) A $2n \times n$ matrix of rank $r=\mathrm{nint}(3n/4)$ generated in the same way as in (2), but with $C_{ii}$ and $S_{ii}$ replaced by zero for $n-r$ random indices $i \in \{1,2,\dots,n\}$.
\item[1'-4'.] (Noisy) Tests (1-4), each perturbed by \verb$1e-10*(randn(2*n,n)+i*randn(2*n,$\newline\verb$n))$.
\end{enumerate}

We ran these tests with $n = \mathrm{nint}(30 \cdot 2^{j/2})$, $j=0,1,\dots,9$.  In all of the tests, we performed the post-processing procedure suggested in Remark~(\ref{postprocessing}).

Tables~\ref{tab:fullrank}-\ref{tab:rankdef} report the scaled residuals $\frac{\|\widehat{A}-A\|_2}{d(A)}$ and scaled orthogonality measures $\frac{\|\widehat{U}_1^*\widehat{U}_1-I\|_2}{u}$, $\frac{\|\widehat{U}_2^*\widehat{U}_2-I\|_2}{u}$, $\frac{\|\widehat{V}_1^*\widehat{V}_1-I\|_2}{u}$ for each test, where $\widehat{A}= \left( \begin{smallmatrix} \widehat{U}_1 \widehat{C} \widehat{V}_1^* \\ \widehat{U}_2 \widehat{S} \widehat{V}_1^* \end{smallmatrix} \right)$, $d(A)$ is given by~(\ref{distpartialiso}), and $u=2^{-53}$ is the unit roundoff.  For comparison, the results obtained with LAPACK's csd function are recorded in Table~\ref{tab:fullrank} as well.  (Results from LAPACK are not shown in Table~\ref{tab:rankdef}, since LAPACK's csd function applies only to full-rank matrices.)   Inspection of Table~\ref{tab:fullrank} reveals that in most of the tests involving full-rank $A$, the residuals and orthogonality measures were closer to zero for Zolo-csd than for LAPACK.

\end{document}